\pdfoutput=1
\documentclass[11pt,table]{amsart}

\usepackage[utf8]{inputenc}
\usepackage[T1]{fontenc}
\usepackage[a4paper,text={475pt,670pt},headsep=8mm, centering]{geometry}
\usepackage[hidelinks,pdfusetitle]{hyperref}
\usepackage[activate={true,nocompatibility},final,tracking=true,kerning=true,factor=1100,stretch=10,shrink=10]{microtype}
\usepackage{braket}    %%%% \bracket
\usepackage{amsthm}    %%%% theorems
\usepackage{amssymb}   %%%% \precurly
\usepackage{mathrsfs}  %%%% \mathsf{L}
\usepackage[shortlabels]{enumitem}
\usepackage{caption}
\usepackage[margin=5pt,justification=centering,labelformat=simple,labelfont=sc]{subcaption}

\usepackage{bbold}

\theoremstyle{plain}
\newtheorem*{theo*}{Theorem}
\newtheorem{theorem}{Theorem}[section]
\newtheorem{corollary}[theorem]{Corollary}
\newtheorem{proposition}[theorem]{Proposition}

\newtheorem{definition}[theorem]{Definition}

\theoremstyle{remark}
\newtheorem{remark}[theorem]{Remark}
\newtheorem*{rem*}{Remark}
\newtheorem{example}[theorem]{Example}

\DeclareMathOperator\spec{Spec}               % supremum
\newcommand{\cb}{\ensuremath{\mathscr{B}}}

\newcommand{\cf}{\ensuremath{\mathscr{F}}}

\newcommand{\ind}[1]{\mathbb{1}_{#1}}
\newcommand{\un}{\mathbb{1}}
\newcommand{\zero}{\mathbb{0}}
\newcommand{\ca}{\ensuremath{\mathscr{A}}}

\usepackage[backend=biber,
	    url=false,
	    doi=false,
	    isbn=false,
	    date=year,
	    maxbibnames=99,
	    giveninits,
	    sortcites=true,
	    sorting=nty,
	    style=numeric]{biblatex}

\renewbibmacro{in:}{}
\newcommand{\abs}[1]{\left\lvert\,#1\,\right\rvert}
\newcommand{\norm}[1]{\left\lVert\,#1\,\right\rVert}

\newcommand{\floor}[1]{\left\lfloor\,#1\,\right\rfloor}
\newcommand{\R}{\ensuremath{\mathbb{R}}}

\newcommand{\C}{\ensuremath{\mathbb{C}}}
\newcommand{\N}{\ensuremath{\mathbb{N}}}

\newcommand{\Diag}{\ensuremath{\mathrm{Diag}}}

\newcommand{\cpa}{\mathcal{P}^\mathrm{Anti}}
\newcommand{\F}{\ensuremath{\mathcal{F}}}
\newcommand{\AF}{\ensuremath{\mathcal{F}^\mathrm{Anti}}}
\newcommand{\kk}{\ensuremath{\mathrm{k}}}
\newcommand{\rd}{\ensuremath{\mathrm{d}}}
\newcommand{\cll}{\ensuremath{\mathcal{L}}}

\newcommand{\cp}{\ensuremath{\mathcal{P}}}

\newcommand{\FF}{\ensuremath{\mathbf{F}}}
\newcommand{\as}{\text{ a.s.}}

\newcommand{\costu}{\ensuremath{C_\mathrm{uni}}} 
\newcommand{\cmir}{\ensuremath{c_\star}}
\newcommand{\cmar}{\ensuremath{c^\star}}
\newcommand{\cmax}{c_{\max}}
\newcommand{\mir}{\ensuremath{R_{e\star}}}
\newcommand{\mar}{\ensuremath{R_e^\star}}
\newcommand{\Cinf}{\ensuremath{C_\star}} 
\newcommand{\Csup}{\ensuremath{C^\star}} 
\newcommand{\oa}{\ensuremath{\Omega_\mathrm{a}}}
\newcommand{\oi}{\ensuremath{\Omega_\mathrm{i}}}
\date{\today}

\author{Jean-François Delmas}
\address{Jean-François Delmas,
  CERMICS, \'{E}cole des Ponts, France}
\email{jean-francois.delmas@enpc.fr}

\author{Dylan Dronnier}
\address{Dylan Dronnier,
  Université de Neuchâtel, Switzerland}
\email{dylan.dronnier@unine.ch}

\author{Pierre-André Zitt}
\address{Pierre-André Zitt, LAMA, Université Gustave Eiffel, France}
\email{pierre-andre.zitt@univ-eiffel.fr}

\newcommand{\loss}{{R_e}}

\newcommand{\Ta}{\ensuremath{T_\mathrm{a}}}
\newcommand{\ka}{\ensuremath{\kk_\mathrm{a}}}
\newcommand{\cfi}{\ensuremath{\cf_\mathrm{inv}}}
\newcommand{\Id}{\ensuremath{\mathrm{Id}}}
\newcommand{\sq}{\ensuremath{\mathcal{A}}}

\usepackage{booktabs}
\usepackage{multirow}
\usepackage{utfsym,bm} % cross, bold maths

\title[On optimal vaccinations]{Optimal vaccinations: Cordons sanitaires,
  reducible population and optimal rays}

\begin{document}

\thanks{This work is partially supported by Labex B\'ezout reference ANR-10-LABX-58}



\begin{abstract}
  We consider the bi-objective problem of allocating doses of a (perfect) vaccine to an
  infinite-dimensional metapopulation in order to minimize simultaneously the vaccination
  cost and the   effective reproduction number~$R_e$, which is defined as the spectral
  radius of the effective next-generation operator.

  In this general framework, we prove that a \emph{cordon sanitaire}, that is, a strategy
  that effectively disconnects the non-vaccinated population, might not be optimal, but it
  is still better than the ``worst'' vaccination strategies. Inspired by graph theory, we
  also compute the minimal cost which ensures that no infection occurs using independent
  sets. Using Frobenius decomposition of the whole population into irreducible
  subpopulations, we give some explicit formulae for optimal (``best'' and ``worst'')
  vaccinations strategies. Eventually, we provide some sufficient conditions for a scaling
  of an optimal strategy to still be optimal.
\end{abstract}

\maketitle

\section{Introduction}

\subsection{Vaccination in metapopulation models}

In metapopulation epidemiological models, the population is composed of $N$
subpopulations labelled $1,\ldots,N$,  of respective sizes $\mu_1, \ldots, \mu_{N}$.
Following \cite{hill-longini-2003}, much of the behaviour of the epidemic may be derived
from the so called next-generation matrix $K = (K_{ij})_{1\leq i,j\leq N}$, where $K_{ij}$
corresponds  to an expected number of secondary infections for people in subpopulation $i$
resulting from a single randomly selected non-vaccinated infectious person in subpopulation
$j$.

A vaccination strategy is represented by a vector $\eta\in \Delta=[0, 1]^N$, where
$\eta_i$ is the \textbf{fraction of non-vaccinated} individuals in the
$i$\textsuperscript{th} subpopulation. In particular, $\eta_i$ is equal to $0$ when the
$i$\textsuperscript{th} subpopulation is fully vaccinated, and $1$ when it is not
vaccinated at all. The  strategy $\un\in \Delta$, with all its  entries equal to  1,
therefore corresponds to an entirely non-vaccinated population. The spectral radius
(\textit{i.e.}, the largest modulus of  the eigenvalues) of $K \cdot \Diag(\eta)$, denoted
$R_e(\eta)$, is referred to as the \emph{effective reproduction number}, and may then be 
interpreted as the expected number of cases directly generated by one typical case where
all non-vaccinated individuals are susceptible to the infection. In particular, we denote
by $R_0=R_e(\un)$ the so-called \emph{basic reproduction number} associated to the
metapopulation epidemiological model. We refer to Section~\ref{sec:ST-TNG} for the
computation of the reproduction number for a wide-class of compartmental metapopulation
models appearing in the literature.

With this interpretation of the reproduction number in mind, it is then natural to
minimize it on the space $\Delta$ under a constraint on the cost $C$. A  natural choice
for the cost function is given by the  uniform cost $\costu(\eta)= \sum_i (1 - \eta_i)
\mu_i$, which corresponds to  the fraction of vaccinated individuals in the population.
This constrained optimization problem appears in most of the literature for designing
efficient vaccination strategies for multiple epidemic situation (SIR/SEIR)
\cite{EpidemicsInHeCairns1989, hill-longini-2003, TheMostEfficiDuijze2016,
  CriticalImmuneMatraj2012, poghotanyan_constrained_2018, OptimalInfluenEnayat2020,
IdentifyingOptZhao2019}. Note that in some of these references, the effective reproduction
number is defined as the spectral radius of the matrix $\Diag(\eta) \cdot K$. Since the
eigenvalues of $\Diag(\eta) \cdot K$ are exactly the eigenvalues of the matrix $K\cdot
\Diag(\eta)$, this actually defines the same function $R_e$.

The goal of this paper is to prove a number of properties of the optimal vaccination
strategies associated to a bi-objective optimization problem with cost function $C$ and
loss function~$R_e$, that shed a light on how to vaccinate in the best possible
way. In previous works~\cite{delmas_infinite-dimensional_2020, ddz-theo}, we introduced
a general kernel framework in which the matrix formulation appears  as a special
finite-dimensional  case.  We state  our results in  this  general framework, but for 
ease of  the presentation,  we shall stick to the matrix formulation in this
introduction. We also refer the interested reader to~\cite{ddz-Re} for a detailed study of
$R_e$ and its convexity property, and to~\cite{ddz-reg} for various examples of
kernels and optimal vaccination strategies.

In our previous work~\cite{ddz-theo}, we assumed only minimal hypothesis on the so-called
loss function whose aims to measure the vulnerability of the population. Here, we choose
to take the effective reproduction number as the loss. We also consider strictly
decreasing cost functions (because vaccinating more people costs more; see
Section~\ref{sec:P-et-AP}). These more restrictive assumptions allow us to simplify some
of the statements made in~\cite{ddz-theo} and to give additional specific results.

In bi-objective optimization, one can identify Pareto (resp.\ anti-Pareto) optimal
vaccinations strategies, informally ``best'' (resp.\ ``worst'') vaccination strategies, in
the sense that every strategy that does strictly better for one objective must do strictly
worse for the other (resp.\ every strategy that does strictly worse for one objective must
do strictly better for the other). We refer to \cite[Section 5]{ddz-theo} for details. We
also  consider the Pareto frontier $\F$ (resp.\ anti-Pareto frontier $\AF$) as the
outcomes $(C(\eta), R_e(\eta))$ of the Pareto (resp.\ anti-Pareto) optimal strategies
$\eta$; see Section~\ref{sec:P-et-AP}. In Figure~\ref{fig:sym-circle-pareto}, we have
plotted in red the Pareto frontier and in a dashed red line the anti-Pareto frontier when
the next-generation matrix is the adjacency matrix of the non-oriented cycle graph with
$N=12$ nodes from Figure~\ref{fig:cycle-graph} and Example~\ref{ex:cycle-graph}; see also
Example~\ref{exple:meta}.

\begin{figure}
  \begin{subfigure}[T]{.5\textwidth}
    \centering
    \includegraphics[page=5]{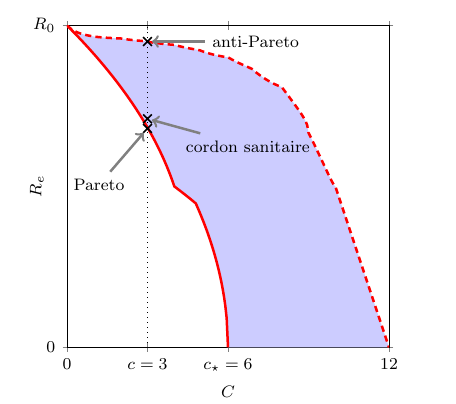}
    \caption{Thick red line: Pareto frontier; Dashed line: anti-Pareto
      frontier; x marker:
    outcomes of various strategies; light blue: all possible outcomes
    $(C(\eta), R_e(\eta))$ for $\eta\in \Delta$.}
    \label{fig:sym-circle-pareto}
  \end{subfigure}%
  \begin{subfigure}[T]{.5\textwidth}
    \centering
    \includegraphics[page=6]{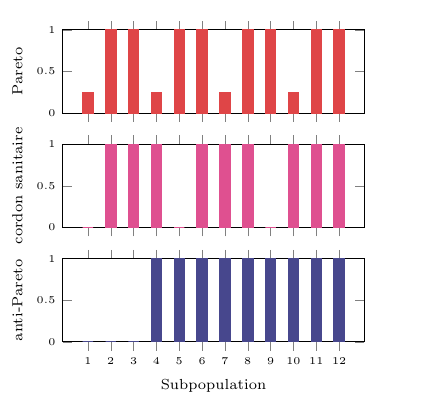}
    \caption{Profile of various strategies.}
    \label{fig:sym-circle-strategies}
  \end{subfigure}
  \caption{Performance of the disconnecting vaccination strategy ``one
    in $4$'' for the non-oriented
    cycle graph with 12 nodes and uniform cost $3$.}
  \label{fig:perf}
\end{figure}

\subsection{A \emph{cordon sanitaire} is not the worst vaccination strategy}

Recall that a matrix $K$ is reducible if there exists a permutation $\sigma$ such that
$(K_{\sigma(i)\sigma(j)})_{i,j}$ is block upper triangular, and irreducible otherwise. A
\emph{cordon sanitaire} is a vaccination strategy $\eta$ such that the effective
next-generation matrix $K \cdot \Diag(\eta)$ is reducible. Informally, such a strategy
splits the effective population in at least two groups, one of which does not infect the
other.

Disconnecting the population by creating a cordon sanitaire is not always the ``best''
choice, that is, it may not be Pareto optimal. However, we prove in
Proposition~\ref{prop:cut} that a cordon sanitaire can never be anti-Pareto optimal; this
result still holds in the general kernel framework, provided that the definition of a
cordon sanitaire is generalized in an appropriate way.

\begin{example}[Non-oriented cycle graph]\label{ex:cycle-graph}
  %%% EXAMPLE
  Suppose that the matrix $K$ is given by the adjacency matrix of the non-oriented cycle
  graph with $N=12$ nodes and $\mu$ is the counting measure; see
  Figure~\ref{fig:cycle-graph} for the graph drawing and Figure~\ref{fig:kernel-cycle} for
  the grayscale representation of its corresponding kernel. For a cost $\costu = 3$,
  there is a \textit{cordon sanitaire} $\eta$ that consists in vaccinating one
  subpopulation in four; see Figure~\ref{fig:cycle-disc} and Figure~\ref{fig:kernel-sep}.
  The effective reproduction number is then equal to $\sqrt{2}$. This strategies performs
  better than the anti-Pareto optimal strategy but it is not Pareto optimal as we can see
  in Figure~\ref{fig:perf}. This example is discussed in detail in
  \cite[Section~2.4]{ddz-reg}.
\end{example}

\begin{figure}
  \begin{subfigure}[T]{.5\textwidth}
    \centering
    \includegraphics[page=1]{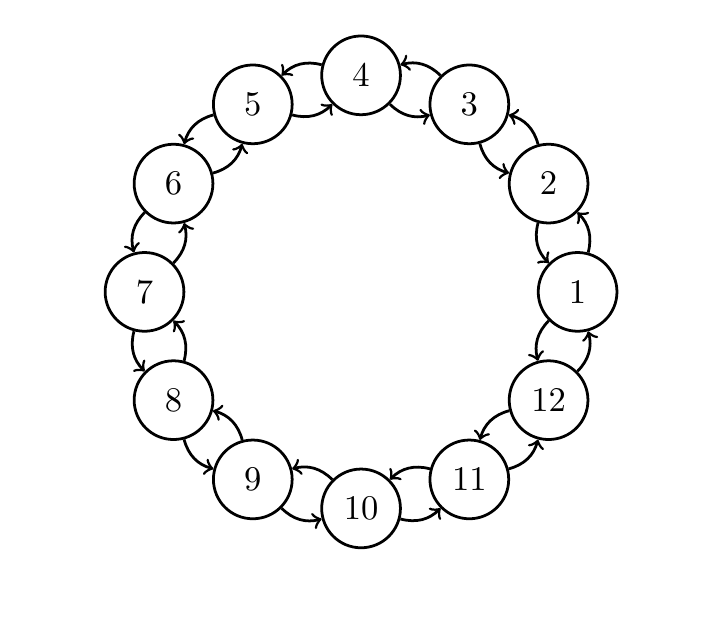}
    \caption{The non-oriented cycle graph.}\label{fig:cycle-graph}
  \end{subfigure}%
  \begin{subfigure}[T]{.5\textwidth}
    \centering
    \includegraphics[page=3]{cordon}
    \caption{Grayscale representation of the corresponding kernel.}
    \label{fig:kernel-cycle}
  \end{subfigure}

  \begin{subfigure}[T]{.5\textwidth}
    \centering
    \includegraphics[page=2]{cordon}
    \caption{Cordon sanitaire corresponding to the
      ``one in $4$'' vaccination strategy. In green the vaccinated subpopulations.}
    \label{fig:cycle-disc}
  \end{subfigure}%
  \begin{subfigure}[T]{.5\textwidth}
    \centering
    \includegraphics[page=4]{cordon}
    \caption{Grayscale representation of the corresponding effective kernel.}
    \label{fig:kernel-sep}
  \end{subfigure}
  \caption{Example of disconnecting vaccination strategy on the
    non-oriented cycle graph with $N=12$ nodes. The kernels are defined on the square
  $[0,12] \times [0,12]$ equipped with the Lebesgue measure.}
  \label{fig:cycle-kern-graph-disc}
\end{figure}

\subsection{Minimal cost required to completely stop the transmission of the disease}

Suppose that the next-generation matrix is symmetric. Then, a vaccination strategy $\eta$
such that $R_e(\eta)=0$ completely stops the transmission of the infection.
Section~\ref{sec:independance} is devoted to the computation of the minimal cost for
achieving this goal. We give in Proposition~\ref{prop:CRe(0)} an explicit
expression of this quantity in the kernel model. When $K$ is the adjacency matrix of a
graph of size $N$, $\mu$ is the counting measure over the set of nodes and the cost is
uniform, this expression is equal to the size of maximal independent sets. We observe this
property in Figure~\ref{fig:sym-circle-pareto} as the size of the maximal independent set
of the non-oriented cycle graph from Example~\ref{ex:cycle-graph} is equal to
$\floor{N/2}$.

\subsection{Reducible case}

When the matrix $K$ happens to be  reducible, up to a relabeling, we may assume that it is
block upper triangular. Denoting by $m$ the number of blocks and $I_1, \ldots, I_m$ the
sets of indices describing the blocks, this means that for  all $\ell>k$ and  $(i,j)\in
I_\ell\times  I_k$, we have $K_{ij} = 0$. In the epidemiological interpretation, this
means that the  populations with indices in $I_k$ never infect the  ones with indices in
$I_\ell$. One may then hope that the study of $R_e$ can be effectively reduced to  the
study of the effective radius  of the square sub-matrices  $(K_{ij})_{i,j\in I_k}$ 
describing the infections within block $I_k$. This is indeed the case, and we give in
Section~\ref{sec:reducible} a complete  picture of the Pareto  and anti-Pareto frontiers
of $R_e$, in terms of  the effective reproduction numbers restricted to each irreducible
component of the infection kernel or matrix. In particular, this allows  a better
understanding of why the anti-Pareto frontier may be discontinuous, while the  Pareto
frontier is always continuous. For the reduction to each irreducible component to be
effective for the  Pareto frontier, one has to assume that the  cost function is
extensive: the cost of vaccinating disjoint subsets of the population is  additive.  Once
more, special  care has to be  taken with the definitions when handling the infinite
dimensional kernel case.

\subsection{Optimal ray}

It is observed by Poghotanyan, Feng, Glasser and Hill in \cite[Theorem
4.3]{poghotanyan_constrained_2018}, that in the finite dimensional case, under an
assumption that ensures the convexity of the function~$R_e$, and for a uniform cost, if
there exists a Pareto optimal strategy $\eta$ with all its entries strictly less than 1,
then all the strategies $\lambda \eta$, with $\lambda\geq 0$ such that $\lambda \eta\in
\Delta$, are Pareto optimal. We give a short proof on the existence of such optimal rays
in Section~\ref{sec:ray} in a general kernel framework,  when  the cost function $C$ is
affine and $R_e$ is convex on $\Delta$.

\subsection{Organization of the paper}

We present in Section~\ref{sec:ST-TNG} different models for which the effective
reproduction number associated to an epidemic model with vaccination can be seen as the
spectral radius of a compact operator. In Section~\ref{sec:settings}, we present the
mathematical framework for the study of the effective reproduction function and the
associated bi-objective problems with a general cost function as well as the Pareto and
anti-Pareto frontiers. Section~\ref{sec:ray-eradicate} is devoted to the description of
optimal vaccination strategies which eradicate the epidemic, and  the possible existence
of optimal rays in the Pareto frontier. Using a Frobenius decomposition of the next
generation kernel in Section~\ref{sec:atom}, we first complete the description of  the
anti-Pareto frontier in the irreducible and  monatomic cases in Section~\ref{sec:q-irr}.
We study in Section~\ref{sec:cordon} the optimality of \emph{cordons sanitaires}
vaccination strategies and show in Section~\ref{sec:reducible} how the optimization
problem may be effectively reduced to the study on subpopulations when the next generation
kernel is reducible.

\section{Generality of  the effective next-generation operator}\label{sec:ST-TNG}

In \cite{delmas_infinite-dimensional_2020, ddz-theo}, we developed a framework that we
call the kernel model where the population is represented as an abstract  measure  space
$(\Omega, \mathscr{F}, \mu)$, with $\mu$ non-zero  $\sigma$-finite measure.  Individuals
are characterized by a trait $x \in \Omega$. The size of the subpopulation with trait
$x$ is given by $\mu(\mathrm{d} x)$. The underlying structure described by this trait can
be very diverse. Typical examples include
spatial position,  social contacts, susceptibility,  infectiousness, characteristics  of
the  immunological response, etc.  The analogue of  the next-generation matrix~$K$ is
the kernel operator defined formally by:
\[
  T_\kk ( g)(x) = \int_\Omega \kk(x,y) \, g(y) \,\mathrm{d}\mu(y),
\]
where the non-negative kernel $\kk$ is defined on $\Omega\times \Omega$ and $\kk(x,y)$
still represents a strength of infection from $y$ to $x$. Vaccination strategies $\eta \,
\colon \, \Omega \to [0,1]$ encode the \textbf{density of non-vaccinated individuals} with
respect to the measure~$\mu$. So, the strategy $\eta=\un$, the constant function equal to
1, corresponds to no vaccination in the population, whereas the strategy  $\eta=\zero$,
the constant function equal to 0, corresponds to all the population being vaccinated.
The measure $\eta(y)\, \mu(\mathrm{d} y)$ may then be understood as an effective
population, giving rise to an effective next-generation operator:
\[
  T_{\kk\eta} ( g)(x) = \int_\Omega \kk(x,y) \, g(y) \, \eta(y) \,
  \mu(\mathrm{d} y).
\]
The effective reproduction number is then defined by $R_e(\eta) = \rho(T_{\kk \eta})$,
where $\rho$ stands for the spectral radius of the operator and $\kk \eta$ for the kernel
$(\kk \eta)(x,y)=\kk(x,y) \eta(y)$.

The results mentioned in the introduction  will be given in  this general framework, which
is  flexible enough to  describe a wide range of  epidemic models from the  literature
including the metapopulation models. In the following of the Section, we  give a few
examples to support this claim. In each of them, the spectral radius of a given explicit
kernel operator appears as a threshold parameter, and the epidemic  either expands or dies
out depending  on  the value  of this parameter. Classical notations are used: $S$ denotes
the proportion of susceptible individuals, $E$  the proportion of those who have been
exposed  to the disease, $I$ the proportion of infected individuals, $R$ the proportion of
removed individuals in the  population. Thus $I(t, x)$ denotes the proportion of the
population with  trait $x\in \Omega$ which is infected at time $t\geq 0$. In the following
examples, the measure $\mu$ is assumed to be a probability measure.

\begin{example}[Metapopulation models]\label{exple:meta}
  %%% EXAMPLE: METAPOPULATION MODELS
  Recall that in metapopulation models, the population is divided into $N \geq 2$
  different subpopulations of respective proportional size $\mu_1, \ldots ,\mu_N$, and
  the reproduction number is given by $R_e(\eta) = \rho(K \cdot\Diag(\eta))$, where $K$ is
  the next generation matrix and $\eta$ belongs to $[0,1]^N$ and gives the proportion of
  non-vaccinated individuals in each subpopulation. To express the function $R_e$ as the
  effective reproduction number of a kernel model, consider the discrete state space
  $\Omega_{\mathrm{d}} = \{ 1, \ldots, N \}$ equipped with the probability measure
  $\mu_{\mathrm{d}}$ defined by $\mu_{\mathrm{d}}(\{i\}) = \mu_i$, and let
  $\kk_{\mathrm{d}}$ denote the discrete kernel on $\Omega_{\mathrm{d}}$ defined by:
  \begin{equation}\label{eq:next-kernel}
    \kk_{\mathrm{d}}(i,j) = K_{ij}/\mu_j.
  \end{equation}
  For all $\eta \in \Delta = [0,1]^N$, the matrix $K\cdot \Diag(\eta)$ is the matrix
  representation of the endomorphism $T_{\kk_{\mathrm{d}} \eta}$ in the canonical basis of
  $\R^N$. In particular, we have: $R_e(\eta) = \rho(T_{\kk \eta}) = \rho(K
  \cdot\Diag(\eta))$.

  In Figure~\ref{fig:kernel-cycle}, we have plotted a kernel on $[0, 1]$ endowed with the
  usual Borel $\sigma$-algebra and the Lebesgue measure. This kernel is equivalent to
  $\kk_{\mathrm{d}}$ when $K$ is the adjacency matrix of the non-oriented cycle graph and
  all subpopulations have the same size.
\end{example}

\begin{example}[An SIR model with nonlinear incidence rate and vital dynamics]
  %%% EXAMPLE
  In  \cite{thieme_global_2011},  Thieme proposed  an  SIR  model in  an
  infinite-dimensional population  structure with a  nonlinear incidence
  rate. The  structure space is  given by  $\Omega$ a compact subset of
  $\R^N$ with nonempty interior equipped  with  the  Lebesgue  measure  denoted  by
  $\mu$. We restrict slightly his assumptions so that the incidence rate
  is a linear  function of the number of susceptible.  Besides, we write
  explicitly  the  equation  giving   the  evolution  of  the  recovered
  compartment.  It  does not  play  a  role  in the  long-time  behavior
  analysis of  the equations made by  Thieme but it helps  to understand
  the model when taking into account the vaccination. The dynamic of the
  epidemic then writes:
  \begin{equation}
    \text{For $ t \geq 0, x \in \Omega$,}\,\, \left\{
      \begin{array}{ll}
	\partial_t S(t,x) = \Lambda(x) - \nu_S(x) S(t,x) -
        S(t,x)\,\int_\Omega f(I(t,y), x,
	y) \, \mu( \mathrm{d}y), \\
	\\
	\partial_t I(t,x) = S(t,x)\, \int_\Omega f(I(t,y), x, y) \, \mu(\mathrm{d}y)
	- (\gamma(x) + \nu_I(x))I(t,x), \\
	\\
	\partial_t R(t,x) = \gamma(x) I(t,x) - \nu_R(x) R(t,x),
      \end{array}
    \right.
  \end{equation}
  where, at location $x \in \Omega$:
  \begin{itemize}%[--]
    \item $\Lambda(x)$ is the rate at which fresh susceptible individuals are recruited,
    \item $\nu_S(x)$, $\nu_I(x)$, $\nu_R(x)$ are the \emph{per capita} death rate of the
      susceptible, infected and recovered individuals respectively,
    \item $\gamma(x)$ is the \emph{per capita} recovery rate of infected individuals,
    \item the integral term describes the incidence at time $t$, \emph{i.e.}, the rate of
      new infections.
  \end{itemize}
  The threshold  parameter identified in  \cite{thieme_global_2011}, that
  plays the  role of the reproduction  number, is given by  the spectral
  radius of the operator $T_\kk$ with the kernel $\kk$ given by:
  \begin{equation*}
    \kk(x,y) = \frac{\Lambda(x)}{\nu_S(x)(\gamma(x) + \nu_I(x))} \partial_I f(0, x,y),
    \quad x,y \in \Omega,
  \end{equation*}
  where $\partial_I f(0, x,y)$, the derivative of $f$ with respect to its first variable
  $I$, is supposed to be non-negative.

  Suppose that individuals at location $x$ are vaccinated with probability $1-\eta(x)$ at
  birth. In the corresponding model, the rate at which susceptible individuals with trait
  $x$ are recruited becomes equal to $\eta(x) \Lambda(x)$ while recovered/immunized
  individuals are recruited at rate $(1-\eta(x)) \Lambda(x)$ at location $x$ so that the
  dynamic of the recovered compartment is given by:
  \begin{equation*}
    \partial_t R(t,x) = (1 - \eta(x)) \Lambda(x) + \gamma(x) I(t,x) - \nu_R(x) R(t,x),
    \qquad x \in \Omega,\, t \geq 0.
  \end{equation*}
  The  threshold parameter  $R_e(\eta)$ is  then given  by the  spectral radius of the
  integral operator  $T_{\eta \kk}$ with kernel $\eta\kk $ given    by
  $(\eta\kk)(x,y)=\eta(x)   \kk(x,y)$. According to Equation~\eqref{eq:r(AB)}, we  have
  $\rho(T_{\eta \kk})  = \rho(T_{\kk \eta})$,  and our framework  can be used for this
  model.

  Under regularity assumptions on the parameters of the model, Thieme proved that if
  $R_e(\eta)$ is greater than $1$, then there exists an endemic equilibrium that attracts
  all the solutions while if $R_e(\eta)$ is smaller than $1$, then $I(t,x)$ converges to $0$
  for all $x \in \Omega$ as $t$ goes to infinity.
\end{example}

\begin{example}[An SEIR model without vital dynamics]
  %%% EXAMPLE SEIR
  In \cite{FinalSizeAndAlmeid2021}, Almeida, Bliman, Nadin and Perthame studied an
  heterogeneous SEIR model where the population is again structured with a bounded subset
  $\Omega \subset \R^N$ with nonempty interior equipped with the Lebesgue measure denoted by $\mu$.
  This time however there is no birth nor death of the individuals.
  The dynamic of the susceptible, exposed, infected and recovered individuals writes:
  \begin{equation}
    \text{For $t \geq 0$, $x \in \Omega$,}\qquad \left\{
      \begin{array}{ll}
	\partial_t S(t,x) = - S(t,x) \,\int_\Omega k(x,y) I(t,y)\, \mu(\mathrm{d}y), \\
	\\
	\partial_t E(t,x) = S(t,x) \, \int_\Omega k(x,y) I(t,y)\,
	\mu(\mathrm{d}y) - \alpha(x) E(t,x), \\
	\\
	\partial_t I(t,x) = \alpha(x) E(t,x) - \gamma(x) I(t,x), \\
	\\
	\partial_t R(t,x) = \gamma(x) I(t,x).
      \end{array}
    \right.
  \end{equation}
  Here, the average incubation rate is denoted by $\alpha(x)$ and the average recovery
  rate by $\gamma(x)$; both quantities may depend upon the trait $x$. The function $k$ is
  the transmission kernel of the disease. In this model, the basic reproduction number is
  given by the spectral radius of the integral operator $T_\kk$ with kernel $\kk=k/\gamma$
  given by:
  \begin{equation}\label{eq:SEIR}
    \kk(x,y) = k(x,y)/\gamma(y).
  \end{equation}
Note that the basic reproduction number does not depend on the  average
incubation rate $\alpha$ as in the one-dimensional SEIR model with
constant population size; see~\cite[Section~2.2]{driessche} with death
rate $d=0$.

  Suppose that, prior to the beginning of the epidemic, the decision maker immunizes a
  density $1 - \eta$ of individuals. According to
  \cite[Section~3.2]{FinalSizeAndAlmeid2021}, the effective reproduction number is given
  by $\rho(T_{\eta \kk})$ which is also equal to $\rho(T_{\kk \eta})$.
  Hence, our model is
  indeed suitable for designing optimal vaccination strategies in this context.
\end{example}

\begin{example}[An SIS model without vital dynamic]\label{exple:sis}
  %%% EXAMPLE
  In \cite{delmas_infinite-dimensional_2020}, generalizing the discrete model of
  Lajmanovich and Yorke~\cite{lajmanovich1976deterministic}, we introduced the following
  heterogeneous SIS model where the population is structured with an abstract probability
  space~$(\Omega, \mathscr{F}, \mu)$:
  \begin{equation}
    \text{For $t \geq 0$, $x \in \Omega$,}\qquad \left\{
      \begin{array}{ll}
	\partial_t S(t,x) = - S(t,x) \,\int_\Omega k(x,y) I(t,y)\, \mu(\mathrm{d}y) + \gamma(x)
	I(t,x), \\
	\\
	\partial_t I(t,x) = S(t,x) \,\int_\Omega k(x,y) I(t,y)\, \mu(\mathrm{d}y) - \gamma(x)
	I(t,x).
      \end{array}
    \right.
  \end{equation}
  The function $\gamma$ is the \emph{per-capita} recovery rate and $k$ is the transmission
  kernel. For this model, $R_e(\eta) = \rho(T_{\kk \eta})$ where $\kk=k/\gamma $ is
  defined by $\kk(x,y) = k(x,y)/\gamma(y)$.

  Suppose that, prior to the beginning of the epidemic, a density $1 - \eta$ of
  individuals is vaccinated with a perfect vaccine. In the same way as for the SEIR model,
  we proved, as $t$ goes to infinity, that if $R_e(\eta)$ is smaller than or equal to $1$,
  then $I(t,\cdot)$ converges to $0$, and, under a connectivity assumption on the kernel
  $k$, that if $R_e(\eta)$ is greater than $1$, then $I(t,\cdot)$ converges to the
  (unique)  positive  endemic equilibrium. This highlights the importance of $R_e$ in the
  design of vaccination strategies.
\end{example}

\section{Setting, notations and previous results}\label{sec:settings}

\subsection{Spaces, operators, spectra}\label{sec:spaces}

All metric spaces~$(S,d)$ are endowed with their Borel~$\sigma$-field denoted by $\cb(S)$.
Let~$(\Omega, \cf, \mu)$ be a measured space, with $\mu$ a $\sigma$-finite positive and
non-zero measure.  For~$f$ and~$g$ real-valued functions defined on~$\Omega$, we
write~$\langle f, g \rangle$ or $\int_\Omega f g \, \mathrm{d} \mu$ for $\int_\Omega f(x)
g(x) \,\mu( \mathrm{d} x)$ whenever the latter is meaningful.  For~$p \in [1, +\infty]$,
we denote by $L^p=L^p( \mu)=L^p(\Omega, \mu)$ the space of real-valued measurable
functions~$g$ defined on~$\Omega$ such that $\norm{g}_p=\left(\int |g|^p \, \mathrm{d}
\mu\right)^{1/p}$ (with the convention that~$\norm{g}_\infty$ is the~$\mu$-essential
supremum of $|g|$) is finite, where functions which agree~$\mu$-a.e.\ are identified.  We
denote by~$L^p_+$ the subset of~$L^p$ of non-negative functions.  We define~$\Delta$ as
the subset of $L^\infty $ of $[0, 1]$-valued measurable functions defined on~$\Omega$. We
denote by $\un$ (resp. $\zero$) the constant function on $\Omega$ equal to $1$ (resp.
$0$); both functions  belong to $\Delta$.

Let~$(E, \norm{\cdot})$ be a complex Banach space. We denote by~$\norm{\cdot}_E$ the
operator norm on~$\cll(E)$ the Banach algebra of linear bounded operators. The
spectrum~$\spec(T)$
of~$T\in \cll(E)$ is the set of~$\lambda\in \C$ such that~$ T - \lambda \mathrm{Id}$ does
not have a bounded inverse, where~$\mathrm{Id}$ is the identity operator on~$E$. Recall
that~$\spec(T)$ is a compact subset of~$\C$, and that the spectral radius of~$T$ is given
by:
\begin{equation}\label{eq:def-rho}
  \rho(T)=\max\{|\lambda|\, \colon\, \lambda \in \spec(T)\}=
  \lim_{n\rightarrow \infty } \norm{T^n}_E^{1/n}.
\end{equation}
The element $\lambda\in \spec(T)$ is an eigenvalue if there exists $x\in E$ such that
$Tx=\lambda x$ and $x\neq 0$.

Recall that the spectrum of a compact operator is finite or countable and has at most one
accumulation point, which is $0$. Furthermore, $0$ belongs to the spectrum of compact
operators  in infinite  dimension.   If $A\in  \cll(E)$  is compact  and $B\in \cll(E)$,
then both $AB$ and $BA$ are compact and:
\begin{equation}\label{eq:r(AB)}
  \rho(AB)=\rho(BA).
\end{equation}

We refer  to \cite{schaefer_banach_1974}  for an introduction  to Banach lattices and
positive operators. We  shall only consider the real Banach lattices $L^p=L^p(\Omega,
\mu)$  for $p\in [1, +\infty ]$  on a measured space $(\Omega, \cf,  \mu)$ with a
$\sigma$-finite  non-zero measure, as well as their complex extension. (Recall that the
norm of an operator on $L^p$ or its natural complex extension is the
same according to~\cite[Corollary~1.3]{complex}). A  bounded operator $A$  is positive if
$A(L^p_+) \subset L^p_+$.  If  $A,  B\in  \cll(L^p)$ and  $A-B$  are positive operators,
then:
\begin{equation}
  \label{eq:r(A)r(B)}
  \rho(A)\geq \rho(B).
\end{equation}

If~$E$ is also a real or complex function space, for~$g \in E$, we denote by~$M_g$ the
multiplication operator  (possibly unbounded) defined by~$M_g(h)=gh$ for all~$h \in E$. If
furthermore $g$ is the  indicator function of a set $A$, we simply write $M_A$ for
$M_{\ind{A}}$.

\subsection{Kernel operators}

We define a \emph{kernel} (resp. \emph{signed kernel}) on~$\Omega$ as a $\R_+$-valued
(resp. $\R$-valued) measurable function defined on~$(\Omega^2, \mathscr{F}^{\otimes 2})$.
For~$f,g$ two non-negative measurable functions defined on~$\Omega$ and~$\kk$ a kernel
on~$\Omega$, we denote by $f\kk g$ the kernel defined by:
\begin{equation}\label{eq:def-fkg}
  f\kk g:(x,y)\mapsto f(x)\, \kk(x,y) g(y).
\end{equation}

For~$p \in (1, +\infty )$, we define the double norm of a signed kernel~$\kk$ on $L^p$ by:
\begin{equation}\label{eq:Lp-integ-cond}
  \norm{\kk}_{p,q}=\left(\int_\Omega\left( \int_\Omega \abs{\kk(x,y)}^q\,
  \mu(\mathrm{d}y)\right)^{p/q} \mu(\mathrm{d}x) \right)^{1/p}
  \quad\text{with~$q$ given by}\quad \frac{1}{p}+\frac{1}{q}=1.
\end{equation}
We say that $\kk$ has a finite double norm, if there exists~$p \in (1, +\infty )$ such
that $\norm{\kk}_{p,q}<+\infty$. To such a kernel $\kk$, we then associate the positive
integral operator~$T_\kk$ on~$L^p$ defined by:
\begin{equation}\label{eq:def-Tkk}
  T_\kk (g) (x) = \int_\Omega \kk (x,y)\, g(y)\,\mu(\mathrm{d}y)
  \quad \text{for } g\in L^p \text{ and } x\in \Omega.
\end{equation}
According to~\cite[p. 293]{grobler}, the operator $T_\kk$ is compact. It is well known and easy to
check that:
\begin{equation}\label{eq:double-norm-norm}
  \norm{ T_\kk }_{L^p}\leq \norm{\kk}_{p,q}.
\end{equation}
We define the \emph{reproduction number} associated to the operator $T_\kk$ as:
\begin{equation}\label{eq:def-R0}
  R_0[\kk]=\rho(T_\kk).
\end{equation}

\subsection{The effective reproduction number~%
  \texorpdfstring{$R_e$}{Re}}\label{sec:weak-topo}

A \emph{vaccination strategy}~$\eta$ of a vaccine with perfect efficiency is an element
of~$\Delta$, where~$\eta(x)$ represents the proportion of \emph{\textbf{non-vaccinated}}
individuals with feature~$x$, so that the constant functions  $\eta=\un$ and
$\eta=\zero$ correspond respectively to no vaccination and complete vaccination. Notice
that~$\eta\, \mathrm{d} \mu$ corresponds in a sense to the effective population. Let $\kk$
be a kernel on $\Omega$ with finite double norm on $L^p$. For~$\eta\in \Delta$, the
operator~$M_\eta$ is bounded on $L^p$, whence the operator $T_{\kk \eta} = T_\kk M_\eta$
is compact. We define the \emph{effective reproduction number} function $R_e[\kk]$
from~$\Delta$ to~$\R_+$ by:
\begin{equation}\label{eq:def-R_e}
  R_e[\kk](\eta)=\rho(T_{\kk\eta}),
\end{equation}
and the corresponding reproduction number is then given by $ R_0[\kk]=R_e[\kk](\un)$. When
there is no risk of confusion on the kernel~$\kk$, we simply write $R_e$ and $R_0$ for the
function~$R_e[\kk]$ and the number~$R_0[\kk]$.

We can see~$\Delta$ as a subset of~$L^\infty  $, and consider the corresponding
\emph{weak-*  topology}: a sequence~$(g_n, n \in \N)$ of elements of~$\Delta$ converges
weakly-* to~$g$ if for all~$h \in L^1 $ we have:
\begin{equation}\label{eq:weak-cv}
  \lim\limits_{n \to \infty} \int_\Omega h g_n \, \mathrm{d}\mu= \int_\Omega h g\, \mathrm{d}\mu.
\end{equation}
The set~$\Delta$ endowed with the weak-* topology is compact and sequentially compact
\cite[Lemma~3.1]{ddz-theo}.  We also recall the properties of the effective reproduction
number given in \cite[Proposition~4.1 and Theorem~4.2]{ddz-theo}.

\begin{proposition}\label{prop:R_e}
  %%% PROPOSITION : PROPRIETE DE R_E
  Let  $\kk$  be  a  finite  double norm  kernel  on  a  measured  space
  $(\Omega,   \cf,  \mu)$   where $\mu$ is  a   $\sigma$-finite  non-zero   measure
  Then, the function~$R_e=R_e[\kk]$ is a continuous function from
  $\Delta$   (endowed  with   the  weak-*   topology)  to~$\R_+$.
  Furthermore,  the  function~$R_e=R_e[\kk]$   satisfies  the  following
  properties:
  \begin{enumerate}[(i)]
  \item\label{prop:a.s.-Re}%
    $R_e(\eta_1)=R_e(\eta_2)$ if~$\eta_1=\eta_2,\, \mu\as$, and~$\eta_1, \eta_2 \in \Delta$,
  \item\label{prop:min_Re}%
    $R_e(\zero) = 0$ and~$R_e(\un) = R_0$,
  \item\label{prop:increase}%
    $R_e(\eta_1) \leq R_e(\eta_2)$ for all~$\eta_1, \eta_2\in \Delta$ such
    that~$\eta_1\leq \eta_2$,
  \item\label{prop:normal}%
    $R_e(\lambda \eta) = \lambda R_e(\eta)$, for all~$\eta \in \Delta$ and~$\lambda \in [0,1]$.
  \end{enumerate}
\end{proposition}

\subsection{Pareto and anti-Pareto frontiers}\label{sec:P-et-AP}

Let $\kk$ be a kernel on $\Omega$ with a finite double norm. We consider the effective
reproduction function $R_e=R_e[\kk]$ defined on $\Delta$ as a loss function.  We quantify
the cost of the vaccination strategy $\eta\in \Delta$ by a function $C: \Delta \rightarrow
\R^+$, and we assume that $C(\un)=0$ (doing nothing costs nothing), $C$ is continuous for
the weak-* topology on $\Delta$ defined in Section~\ref{sec:weak-topo} and
\emph{decreasing} (doing more costs strictly more),  that is, for any $\eta_1, \eta_2\in
\Delta$:
\[
  \eta_1\leq \eta_2 \quad\text{and}\quad
  \mu(\eta_1<\eta_2)>0
  \,\implies \, C(\eta_1)> C(\eta_2).
\]

For example, when the measure $\mu$ is finite, the uniform cost function:
\begin{equation}\label{eq:def-C}
  \costu(\eta)=\int_\Omega (1-\eta)\, \rd \mu.
\end{equation}
is continuous and decreasing on $\Delta$ (recall that $1-\eta$ represents the  proportion
of the population which has  been vaccinated when using the strategy $\eta$.)

\newcommand{\missing}{\textcolor{purple}{\hspace{3em}\usym{2717}}}
\newcommand{\mc}[1]{\multicolumn{2}{c}{#1}}
\begin{table}[t]
   \centering
   \footnotesize
  \begin{tabular}{ @{\extracolsep{\fill}} lll}
    \toprule
    & ``Best'' vaccinations & `` Worst'' vaccinations \\
    \midrule
    Optimization problem  &
    Pb~\eqref{eq:bi-min}: $\min_\Delta(C,R_e)$ &
    Pb~\eqref{eq:bi-max}: $\max_\Delta(C,R_e)$ \\
    \midrule
    \multirow{4}{ 12em}
     {Opt.\ cost for a given loss
    defined on $[0, R_0]$, with $R_0:\,=\max_\Delta R_e = R_e(\un)$.
    }
    & $\Cinf (\ell):\,=\min_{\, R_e\leq  \ell}\,  C$.
    & $\Csup (\ell):\, =\max_{\, R_e\geq  \ell}\,  C$.
    \\
    & $\Cinf$ is continuous.
    & \missing \\
    &  $\Cinf$ is  decreasing.
    & $\Csup$ is  decreasing.\\
    & $\Cinf(R_0)=0$ and $\cmir\!:\,=\Cinf(0)$.
    &$\Csup(0)=\cmax$ and $\cmar\!:\,= \Csup(R_0)$.
    \\
    \midrule
    \multirow{5}{12em}
    {Opt.\ loss for a given cost
    defined on $[0, \cmax]$, with
       $\cmax:\,=\max_\Delta C = C(\zero)$.}
    & $\mir (c):\,=\min_{\, C\leq  c}\,  R_e$.
    & $\mar (c):\, =\max_{\, C\geq  c}\,  R_e$.\\
    & $\mir$ is continuous.
    & $\mar$ is continuous. \\
    & $\mir$ is  decreasing on $[0, \cmir]$.
    & \missing \\
    &$\mir=0$ on $[\cmir, \cmax]$.
    & $\mar =R_0$  on $[0, \cmar]$. \\
& $\mir(0)=R_0$. & $\mar (\cmax)=0$. \\
    \midrule
\multirow{2}{12em}{Inverse formula }
    & $\mir\circ \Cinf =\Id\,$ on $[0, R_0]$.
    & $\mar\circ \Csup =\Id\,$ on $[0, R_0]$.\\
    & $\Cinf\circ \mir =\Id\, $ on $[0, \cmir].$
    & \missing \\
    \midrule
\multirow{4}{12em}{Optimal strategies}
    & $\cp:\,=\set{C=\Cinf\circ R_e} \cap \set{R_e = \mir \circ C }$
     & $\cpa:\,=\set{C=\Csup\circ R_e} \cap \set{R_e = \mar \circ C }$ \\
    &  $\phantom{\cp:} =\set{C=\Cinf\circ R_e}$
     &  $\phantom{\cpa:} =\set{C=\Csup\circ R_e}$ \\
     &  $\phantom{\cp:} =   \set{R_e=\mir \circ C, \, C\leq  \cmir}$.
             &   $\phantom{\cpa:} =\set{R_e=\mar \circ
                   C, \, C\geq  \cmar}$. \\
    & $\cp$ is compact. & \missing \\
    \midrule
    Range of cost/loss
    &  \mc{ $\sq :\,=[0, \cmax]\times [0, R_0]$}  \\
\midrule
   \multirow{3}{12em}{Possible outcomes}
    &  \mc{ $\FF :\,=(C, R_e)(\Delta)\phantom{\sq\, \colon\,\mir(c)\leq \ell \leq \mar(c)} $}  \\
    & \mc{ $ \phantom{\FF:} = \set{(c, \ell) \in \sq\, \colon\,
      \mir(c)\leq \ell \leq \mar(c)}$}\\
    & \mc{ $ \phantom{\FF:} = \set{(c, \ell) \in \sq\, \colon\,    \Cinf(\ell)\leq c \leq  \Csup(\ell)}.$} \\
\midrule
\multirow{4}{12em}{Optimal frontier}
    & $\F :\,=(C, R_e)(\cp)$
    & $\AF:\,= (C, R_e)(\cpa)$\\
    &$\phantom{\F:}=(\Cinf, \Id)([0, R_0])$
    & $ \phantom{\AF:}=(\Csup, \Id)([0, R_0])$. \\
    & $\phantom{\F:}=(\Id, \mir)([0, \cmir]) $. &\missing \\
    & $\F$ is connected and compact. &   \missing  \\
    \bottomrule
  \end{tabular}%

  \medskip

  {\small
 The missing results, indicated by
    \textcolor{purple}{\usym{2717}},  will be further
    completed under some additional conditions on the kernel $\kk$
    (see Proposition~\ref{prop:k>0-c} for $\kk$ positive and
    Corollary~\ref{cor:monat} for $\kk$ monatomic).}

  \caption{%
    Summary of notation and results for the bi-objective problems.}

  \label{tab:not-PAP}
\end{table}

In \cite{ddz-theo}, we formalized and study the problem of optimal allocation strategies
for a perfect vaccine. This question may be viewed as a bi-objective minimization problem,
where one tries to minimize simultaneously the cost of the vaccination and its loss given
by the corresponding effective reproduction number:
\begin{equation}
  \label{eq:bi-min}
  \min_{\Delta} (C,R_e).
\end{equation}
Let us now briefly summarize the results from~\cite{ddz-theo}. For the reader's
convenience we also collect the main points in Table~\ref{tab:not-PAP}, and provide plots
of typical Pareto and anti-Pareto frontiers in  Figure~\ref{fig:generic_frontiers}.
Note that Assumptions 4 and 5 in~\cite{ddz-theo} hold thanks to
\cite[Lemma~5.13]{ddz-theo}. By definition, we have $R_0=\max_\Delta \, R_e$ and we set $
\cmax=\max_\Delta C$ which is positive as $C$ is decreasing (and $\mu$ non-zero)  and
finite as $C$ is continuous and $\Delta$ compact. Related to the minimization
problem~\eqref{eq:bi-min}, we shall consider~$\mir$ the \emph{optimal loss} function
and~$\Cinf$ the \emph{optimal cost} function defined by:
\begin{align*}
  \mir (c) &= \min \, \set{ \loss(\eta) \, \colon \, \eta \in \Delta,\,
  C(\eta) \leq c }\quad\text{for $c\in [0,\cmax]$},\\
    \Cinf(\ell) &= \min \, \set{ C(\eta) \, \colon \, \eta \in \Delta,\, \loss(\eta) \leq
    \ell } \quad \text{for $\ell \in [0,R_0]$}.
\end{align*}
We have $\Cinf(R_0)=0$ and $\mir(0)=R_0$ since $C$ is decreasing. For convenience, we
    write $\cmir$ for the minimal cost such that $\mir$ vanishes:
\begin{equation}\label{eq:def-cmir}
  \cmir=\Cinf(0).
\end{equation}
The function $\mir$ is continuous, decreasing on $[0, \cmir]$ and zero on $[\cmir,
    1]$;
    the function $\Cinf$ is continuous and decreasing on $[0, R_0]$; and the functions
    $\mir$ and $\Cinf$ are the inverse of each other, that is, $\mir \circ
    \Cinf(\ell)=\ell$ for $\ell\in [0, R_0]$ and $\Cinf \circ \mir( c)=c$ for $c\in [0,
    \cmir]$.

We define the Pareto optimal strategies $\cp$ as the ``best'' solutions of the
minimization problem~\eqref{eq:bi-min} (we refer to \cite{ddz-theo} for a precise
justification of this terminology):
\[
  \cp=\left\{\eta\in \Delta\, \colon\, C(\eta)=\Cinf( R_e(\eta))
    \quad\text{and}\quad
    R_e(\eta) = \mir( C(\eta)) \right\}.
\]
We have in fact the following representation of the Pareto optimal strategies:
\begin{align*}
  \cp
  &=\left\{\eta\in \Delta\, \colon\, C(\eta)=\Cinf(R_e(\eta))\right\}\\
&  = \left\{\eta\in \Delta\, \colon\, R_e(\eta)=\mir(C(\eta))
    \quad\text{and}\quad C(\eta)\leq  \cmir
  \right\}.
\end{align*}
The Pareto frontier is defined as the outcomes of the Pareto optimal strategies:
\[
  \F=\left\{(C(\eta), R_e(\eta))\, \colon\, \eta\in \cp \right\}.
\]
The set $\cp$ is a non empty compact (for the weak topology) in $\Delta$ and furthermore
the Pareto frontier can be easily represented using the graph of the optimal loss function
or cost function:
\[
  \F
  = \{(\Cinf(\ell), \ell) \, \colon \, \ell \in [0,R_0]\}
  = \{(c, \mir(c)) \, \colon \, c \in [0,\cmir]\}.
\]
It is also of interest to consider the ``worst'' strategies
which can be viewed as solutions to the bi-objective maximization
problem:
\begin{equation}\label{eq:bi-max}
  \max_{\Delta} (C,R_e).
\end{equation}
The next results can be found in \cite[Propositions~5.8
and~5.9]{ddz-theo}. Note therein  that Assumption 6 holds in
general but that Assumption 7 holds under the stronger condition
that the kernel $\kk$ is monatomic; see Section~5.4.2.
Related to the maximization problem~\eqref{eq:bi-max}, we shall
consider~$\mar$ the \emph{optimal loss} function and~$\Csup$ the
\emph{optimal cost} function defined by:
\begin{align*}
\mar (c) &= \max \, \set{ \loss(\eta) \, \colon \, \eta \in \Delta,\,
  C(\eta) \geq c }\quad\text{for $c\in [0,\cmax]$},\\
  \Csup(\ell) &= \max \, \set{ C(\eta) \, \colon \, \eta \in \Delta,\,
    \loss(\eta) \geq \ell } \quad \text{for $\ell \in [0,R_0]$}.
\end{align*}
We have $\Csup(0)=\cmax$ and $\mar(\cmax)=0$ since $C$ is decreasing and
$C(\zero)=\cmax$. Since, for $\varepsilon\in (0, 1)$ we have
$C(\varepsilon \un )<\cmax$ as $C$ is decreasing and
$R_e(\varepsilon \un)=\varepsilon R_0>0$, we deduce that
$\Csup(0+)=\cmax$. For convenience, we write $\cmar$ for the maximal
cost of totally inefficient strategies:
\begin{equation}
    \label{eq:def-cmar}
  \cmar = \Csup(R_0)=\max \{ c \in [0,\cmax] \, \colon \, \mar(c) = R_0 \}.
\end{equation}
The function $\Csup$ is decreasing on $[0, R_0]$;
the function $\mar$ is constant equal to $R_0$ on $[0, \cmar]$; we have
$\mar \circ \Csup(\ell)=\ell$ for $\ell\in [0, R_0]$.
  This latter
property implies that the function $\mar$ is continuous.

We define the anti-Pareto optimal strategies $\cpa$ as the ``worst''
strategies, that is solutions of
the maximization problem~\eqref{eq:bi-max}:
\[
 \cpa
 =\left\{\eta\in \Delta\, \colon\, C(\eta)=\Csup(R_e(\eta))
\quad   \text{and}\quad
   R_e(\eta)=\mar(C(\eta)) 
\right\}.
\]
We have in fact the following representation of the anti-Pareto optimal strategies:
\begin{align*}
  \cpa
  &=\left\{\eta\in \Delta\, \colon\, C(\eta)=\Csup(R_e(\eta))\right\}\\
&  =\left\{\eta\in \Delta\, \colon\, R_e(\eta)=\mar(C(\eta)) 
    \quad\text{and}\quad C(\eta)\geq  \cmar\right\}.
\end{align*}
The anti-Pareto frontier is defined as the outcomes of the anti-Pareto optimal strategies:
\[
  \AF=\left\{(C(\eta), R_e(\eta))\, \colon\, \eta\in \cpa \right\}.
\]
The set $\cpa$ is non empty
and furthermore the Pareto frontier can be easily represented using the
graph of the optimal cost function:
\begin{equation}
   \label{eq:FL=L*}
  \AF
   = \{(\Csup(\ell), \ell) \, \colon \, \ell \in [0,R_0]\}.
 \end{equation}
\medskip

We also have that the feasible region or set of possible outcomes for $(C, R_e)$:
\[
  \FF=\left\{(C(\eta), R_e(\eta))\, \colon \, \eta\in \Delta \right\}
\]
is compact,
  path connected, and its complement is connected in $\R^2$. It is the whole region
  between the graphs of the one-dimensional value functions:
  \begin{align*}
    \FF &= \{ (c,\ell) \in [0, \cmax]\times [0, R_0]\,\colon\,
    \mir(c) \leq \ell \leq \mar(c) \} \\
	&= \{ (c,\ell) \in [0, \cmax]\times [0, R_0] \,\colon\,
	\Cinf(\ell) \leq c \leq \Csup(\ell)\}.
  \end{align*}
We plotted in Figure~\ref{fig:generic_frontiers} the typical Pareto and anti-Pareto
frontiers for a general kernel (notice the anti-Pareto frontier is not connected \textit{a
  priori}). In Section~\ref{sec:reg-AF},  we check that reducibility conditions on
the kernel $\kk$ provide further properties on the frontiers.

\section{Optimal ray and optimal strategies which eradicate the epidemic}
\label{sec:ray-eradicate}

We introduced in Section~\ref{sec:P-et-AP} the bi-objective minimization/maximization
problems, where one tries to minimize/maximize simultaneously the cost of the vaccination
and the effective reproduction number. In Section~\ref{sec:ray}, we derive the existence
of Pareto optimal rays as soon as there exists a Pareto optimal strategy uniformly
strictly bounded from above by $1$; and in Section~\ref{sec:independance} we give a
characterization of $\cmir$ using the notion of independent set from graph
theory.

\subsection{Optimal ray}\label{sec:ray}

If the loss function $R_e$ is convex  and if the cost function is affine, then the set
$\mathcal{P}$ of Pareto optimal strategies may contain a non-trivial optimal ray
$\{\lambda \eta\, \colon\, \lambda\in [0, 1]\}$. This optimal ray has already been
observed in finite dimension \cite{poghotanyan_constrained_2018}. We also refer
to~\cite{ddz-Re} for sufficient condition on the kernel $\kk$ for the function $R_e[\kk]$
to be convex or concave.

\begin{proposition}[Optimal ray]\label{prop:critical-ray}
  %%% PROPOSITION: OPTIMAL RAY
  Suppose that the cost function $C$ takes the form:
  \[
    C(\eta) = \cmax - \int_\Omega \eta c \, \rd \mu
    \quad\text{with}\quad
    \cmax= \int_\Omega  c \, \rd \mu,
  \]
  for a positive function $c\in L^1$,
    and that the loss
    function $R_e[\kk]$, with $\kk$ a finite double norm kernel, is convex.
    If $\eta_\star \in \mathcal{P}$ is a Pareto optimal strategy that
    satisfies $\eta_\star<1$, $\mu$-a.e., then, for all  $\lambda\geq 0$, the strategy
    $\lambda    \eta_\star$   is    Pareto    optimal    as   soon    as
    $\lambda \eta_\star\in \Delta$.

    In particular, the Pareto frontier contains the segment joining the points of
    coordinates
    $(\cmax,0)$ and $(C(\eta_\star/\sup\eta_\star), \loss( \eta _\star/\sup\eta_\star))$.
    We also have $\cmir=\cmax$.
  \end{proposition}

\begin{remark}
  %% REMARK
  Suppose that $C$ takes the form given in the Proposition and that $R_e[\kk]$, with $\kk$
  a finite double norm kernel, is concave. With a similar proof (but for the last part
  which has to be replaced by the fact that $\Csup(0+)=\cmax$ as the set of anti-Pareto
  optimal strategies might not be closed), it is easy to get that if $\eta^\star$ is
  anti-Pareto optimal such that $\eta^\star<1$ $\mu$-a.e., then,  for all $\lambda\geq 0$,
  the strategy $\lambda \eta^\star$ is anti-Pareto optimal as soon as $\lambda
  \eta^\star\in \Delta$.
\end{remark}

\begin{proof}[Proof of Proposition~\ref{prop:critical-ray}]
  Assume that $\eta_\star\in\mathcal{P}$ satifies $\eta_\star <1$ $\mu$-a.e., and
  $\xi_\star\in\Delta$ is a multiple of $\eta_\star$, say
  $\xi_\star=\lambda\eta_\star$.  Assume for now that $\lambda>0$.
  Our goal is to prove that $\xi_\star$ is Pareto optimal.  Let
  $\xi\in\Delta$ be such that $\loss(\xi)\leq \loss(\xi_\star)$: by
  \cite[Proposition~5.5~(ii)]{ddz-theo}, it is enough to show that
  necessarily, $C(\xi)\geq C(\xi_\star)$, or equivalently that
  $\int_\Omega \xi c \,\rd\mu \leq \int_\Omega \xi_\star c\,\rd\mu$.

  To use the optimality of $\eta_\star$, we construct an auxiliary strategy:
  \[ \eta_n = \min\big( (1-n^{-1})\eta_\star + n^{-1} \eta ; 1\big),\]
  where $n\in \N^*$ and  $\eta = \xi/\lambda$ (note that $\eta\notin\Delta$ in general).
  By monotony, convexity and homogeneity of $\loss$, and the fact that $\loss(\xi)\leq \loss(\xi_\star)$ by hypothesis,  we get:
  \begin{align*}
    \loss(\eta_n) &\leq (1-n^{-1}) \loss(\eta_\star) + n^{-1} \loss(\eta) \\
 %                 & = (1- n^{-1}) \loss(\eta_\star) + \frac{1}{n\lambda} \loss (\xi) \\
                  &\leq (1-n^{-1}) \loss(\eta_\star) + \frac{1}{n\lambda} \loss(\xi_\star) \\
                  &= \loss(\eta_\star).
  \end{align*}
  Since $\eta_\star$ is optimal, this implies $C(\eta_n)\geq C(\eta_\star)$, so
  $  \int_\Omega \eta_\star c \,\rd\mu
  \geq  \int_\Omega \eta_n  c \,\rd\mu $.
  We now compute the right hand side, defining $u_n = (1-n^{-1})\eta_\star +
  n^{-1}\eta$, we get:
  \begin{align*}
    \int_\Omega \eta_\star c\,\rd\mu \geq\int_\Omega \eta_n c\,\rd\mu
%    &= \int u_n \ind{u_n\leq 1} c\,\rd\mu + \int \ind{u_n>1} c\,\rd\mu \\
    &= \int_\Omega u_n c\,\rd\mu - \int_\Omega(u_n - 1) \ind{\{u_n>1\}} c\,\rd\mu \\
    &= (1-n^{-1}) \int_\Omega \eta_\star c\,\rd\mu  + n^{-1} \int_\Omega \eta c\,\rd\mu  -
      \int_\Omega(u_n - 1) \ind{\{u_n>1\}} c\,\rd\mu.
  \end{align*}
Rearranging the terms,  we arrive at:
\[
  \int _\Omega \eta c\,\rd\mu
  \leq \int_\Omega \eta_\star c\,\rd\mu + n\int_\Omega(u_n - 1)
  \ind{\{u_n>1\}} c\,\rd\mu.
\]
Elementary computations give that:
\[
0\leq   n(u_n-1) \ind{\{u_n>1\}}\leq \eta \ind{\{n< (\eta -\eta^\star)/(1
  -\eta^\star)\}}.
\]
Since $\mu$-a.e.  $\eta^\star<1$, this implies that
$\mu$-a.e. $\lim_{n\rightarrow \infty }n(u_n -1)
\ind{\{u_n>1\}}=0$. By dominated convergence, we obtain
$\lim_{n\rightarrow\infty } n\int_\Omega(u_n - 1) \ind{\{u_n>1\}}
c\,\rd\mu=0$ and thus:
\[
  \int_\Omega \eta c\,\rd\mu \leq \int_\Omega \eta_\star c\,\rd\mu,
\]
and, multiplying by $\lambda$, we get $\int_\Omega \xi c\,\rd\mu \leq
\int_\Omega \xi_\star c\,\rd\mu$, as claimed.
  Finally, the statement still holds for $\xi_\star = 0$ by letting
  $\lambda$ go down  to zero and  using the fact that  the Pareto optimal
  set is closed \cite[Corollary~5.7]{ddz-theo}.
\end{proof}

\subsection{A characterization of \texorpdfstring{$\cmir=\Cinf(0)$}{C*(0)}
  when the support of \texorpdfstring{$\kk$}{k} is symmetric}
\label{sec:independance}

We characterize the Pareto optimal strategies which minimize $R_e$ when the kernel $\kk$
has a symmetric support, and get a very simple representation of
$\Cinf(0)$ when $\mu$ is finite and the
cost is uniform.

\medskip

Let us first recall a notion from graph theory. If $G=(V,E)$ is an
non-oriented graph with vertices set $V$ and edge set $E$, an
\emph{independent} set of $G$ is a subset $A\subset V$ of vertices
which are pairwise not adjacent, that is, $i,j\in A$ implies
$i j\not\in E$.

Following~\cite{hladky_independent_2020},
we generalize this definition to kernels.

\begin{definition}[Independent sets for kernels]\label{def:indep-number}
  %%% DEFINITION : INDEPENDENT SET FOR KERNEL
  Let $\kk$ be a kernel on $\Omega$. A measurable set $A \in \cf$ is an independent set of
  $\kk$ if $\kk=0$ $\mu^{\otimes 2}$-a.e.\ on $A\times A$.
\end{definition}

In the following result, we prove that ``maximal'' independent sets provide optimal 
Pareto strategies for the loss function $R_e$  and the cost function $C$. This property is
illustrated in Figure~\ref{fig:perf} with  the uniform  cost $C=\costu$ given
by~\eqref{eq:def-C},  where  the Pareto  frontier of  the non-oriented cycle graph from
Example~\ref{ex:cycle-graph}, with $N=12$, is plotted; it is possible to prevent
infections without vaccinating the whole population as $\cmir=1/2<1=\cmax$.

\begin{proposition}\label{prop:CRe(0)}
  %%% PROPOSITION : C*(0) = alpha
  Let $\kk$ be a finite double norm kernel on $\Omega$ such that its support, $\{\kk>0\}$,
  is a symmetric subset of $\Omega^2$ a.e. We have:
  \begin{equation}\label{eq:cmir} \cmir= \Cinf(0) = \min \{C(\ind{A})\,\colon\, \text{$A$
    is an independent set of $\kk $}\}.
  \end{equation}
  Furthermore if $\eta_\star$ is Pareto optimal such that $R_e[\kk](\eta_\star)=0$, then
  $\{\eta_\star>0\}$ is an independent set, $\eta_\star=\ind{\{\eta_\star>0\}}$ a.e.\ and
  $\cmir=C(\ind{\{\eta_\star>0\}})$.
\end{proposition}

\begin{proof}
  Let  $A$ be  an independent  set.  The  effective reproduction  number
  obviously     vanishes    for     the     strategy    $\ind{A}$     as
  $(T_{\kk \ind{A}})^2=T_\kk\,  T_{\ind{A} \kk \ind{A}}=0$.   This
  gives:
\begin{equation}
   \label{eq:cmir2}
    \cmir\leq  \inf \{C(\ind{A})\,\colon\,  \text{$A$ is  an independent
      set of $\kk $}\}.
  \end{equation}

\medskip

Now, let $\eta  \in \Delta$ be such that $R_e[\kk](\eta)  = 0$. We shall
prove    that   $\{    \eta>0\}$   is    an   independent    set.    Let
$f\in L^1\cap L^\infty  $ such that $0<f\leq 1$. Notice  that $f\in L^r$
for   all   $r\in   [1,   +\infty  ]$.    Let   $\varepsilon>0$.   Since
$\kk \eta \geq \varepsilon \kk_\varepsilon$, with
$\kk_\varepsilon= (\eta f)\, \ind{\{\kk \geq \varepsilon\}}
  \,          (\eta          f)$, that is:
\[
  \kk_\varepsilon(x,y)= (\eta f)(x)\, \ind{\{\kk(x,y) \geq \varepsilon\}}
  \,          (\eta          f)(y),
\]
we  get  that  $T_{\kk\eta}-  \varepsilon  T_{\kk_\varepsilon  }$  is  a
positive    operator,   and    deduce   from~\eqref{eq:r(A)r(B)}    that
$\varepsilon      \rho(T_{\kk_\varepsilon})=      \rho      (\varepsilon
T_{\kk_\varepsilon})    \leq    \rho     (T_{\kk\eta})=0$    and    thus
$R_0[\kk_\varepsilon]=0$.                                            Set
$\kk'  = (\eta  f)\,  \ind{\set{\kk>0}}\, (\eta  f)$,  which has  finite
double            norm            in            $L^p$.             Since
$\lim_{\varepsilon\rightarrow   0+}\norm{\kk_\varepsilon   -   \kk'}_{p,
  q}=0$,   we  deduce   from  \cite[Proposition~4.3]{ddz-theo}   on  the
stability                 of                  $R_e$                 that
$R_0[\kk']          =          \lim_{\varepsilon\rightarrow          0+}
R_0[\kk_\varepsilon]=0$. As the support of $\kk$ is symmetric, we deduce
that the non-negative  kernel $\kk'$ is symmetric. Since  $f\in L^2$, we
deduce  that $\kk'$  has  finite  double norm  on  $L^2$.  According  to
Theorem~4.2.15 and  Problem 2.2.9~p.~49 in \cite{Davies07},  we get that
the integral operator $T_{\kk'}$ on  $L^p$ and the integral operator $T$
on $L^2$ with (the same) kernel  $\kk'$ have the same spectrum, and thus
their  spectral radius  is zero.  Since  $T$ is  self-adjoint with  zero
spectral radius, we deduce that $T=0$  and thus a.e. $\kk'=0$. Since $f$
is     positive,     we     deduce      that     $\kk=0$     a.e.     on
$\{\eta>0\}\times \{\eta>0\}$,  and thus $\{\eta>0\}$ is  an independent
set.

We now prove that the  inequality in~\eqref{eq:cmir2} is an equality and
that  the infimum  is reached.   Let  $\eta_\star$ be  a Pareto  optimal
strategy   such   that   $    R_e[\kk](\eta_\star)   =   0$   and   thus
$\cmir=C(\eta_\star)$.   We  deduce  from  the  previous  argument  that
$\{\eta_\star>0\}$     is    an     independent     set;    and     thus
$  R_e[\kk](\ind{\{\eta_\star>0\}})=0$.   Using   the  monotonicity  and
continuity     of     the     cost     function,     we     get     that
$C(\eta_\star)        \geq       C(\ind{\{\eta_\star>0\}})$        since
$\eta_\star\leq    \ind{\{\eta_\star>0\}}$.      This    implies    that
$\ind{\{\eta_\star>0\}}$    is    Pareto     optimal    as    well    as
$C(\eta_\star) = C(\ind{\{\eta_\star>0\}})$.  This gives the claim.

Using the monotonicity of $C$, we also deduce from the equality
$C(\eta_\star) = C(\ind{\{\eta_\star>0\}})$
that a.e.  $\eta_\star=\ind{\{\eta_\star>0\}}$. This ends the proof.
\end{proof}

\begin{remark}[On the independence number]\label{rem:indep}
  %%% REMARK
  The \emph{independence number} of a graph $G$, denoted by $\alpha(G)$, is the maximum of
  $\sharp A$, over all the independent sets $A$ of $G$. Similarly, if $\mu$ is
  a finite measure, we can  define the independence number $\alpha(\kk )$ of the kernel
  $\kk $ by:
  \[
    \alpha(\kk ) = \sup\{ \mu(A)\,\colon\, \text{$A$ is an independent set of
  $\kk $}\},
  \]
  and we say that $A$ is a maximal independent set for $\kk$ if
  $\mu(A)=\alpha(\kk)$.
  Consider the uniform cost $C=\costu $ given by~\eqref{eq:def-C} and a finite double norm 
  kernel  $\kk$  on $\Omega$ such that its support, $\{\kk>0\}$, is a symmetric subset of
  $\Omega^2$ a.e. Then,  we deduce from Proposition~\ref{prop:CRe(0)}, that any Pareto
  optimal strategy $\ind{A_\star}$ for the loss $R_e[\kk]$  corresponds to a maximal
  independent set $A_\star$ of $\kk$ and \emph{vice versa}. In particular, we have:
  \[ \cmir= \Cinf(0) = C( \ind{A_\star})= \cmax -\alpha(\kk). \]
\end{remark}

\section{Atomic decomposition and  cordons sanitaires}
\label{sec:reg-AF}

Following~\cite{schwartz61} and the presentation given in~\cite{ddz-Re}, we recall the
decomposition of the kernel into its irreducible components in Section~\ref{sec:atom}.
Then, in Section~\ref{sec:q-irr}, we complete the properties related to the anti-Pareto
frontier  for kernels having only one irreducible component. We prove in
Section~\ref{sec:cordon} that creating a \emph{cordon sanitaire} is not anti-Pareto 
optimal. Finally, considering reducible kernels in Section~\ref{sec:reducible}, we provide
a decomposition of the optimal cost and loss functions (related to the anti-Pareto and 
Pareto frontiers) by considering  the  corresponding optimization problems on the
irreducible components.

\subsection{Atomic decomposition}\label{sec:atom}

We  follow the  presentation in~\cite[Section~5]{ddz-Re}  on the  atomic decomposition of
positive compact operator and Remark 5.2  therein for the  particular case of integral
operators; see also the  references therein for  further results. Let $\kk$  be a kernel
on  $\Omega$ with a finite double norm.   For $A, B\in \cf$, we write  $A\subset B$ a.e.\
if $\mu(B^c  \cap  A)=0$  and  $A=B$   a.e.\  if  $A\subset  B$  a.e.\  and $B\subset A$
a.e. For  $A, B\in  \cf$, $x\in  \Omega$, we  simply write $\kk(x,A)=\int_{A} \kk(x,y)\,
\mu(\rd y)$, $ \kk(B,x)=\int_{ B} \kk(z,x)\, \mu(\rd z)$ and:
\[
  \kk(B, A)=
  \int_{B \times A} \kk(z,y)\, \mu(\rd z) \mu(\rd y).
\]
A   set  $A\in   \cf$  is   called  \emph{$\kk$-invariant},   or  simply \emph{invariant}
when there  is no  ambiguity on  the kernel  $\kk$, if $\kk(A^c,  A)=0$.  In  the
epidemiological  setting,  the  set  $A$  is invariant if the  sub-population $A$ does not
infect the sub-population $A^c$.  The kernel $\kk$  is \emph{irreducible} (or
\emph{connected}) if any invariant set $A$ is such  that $\mu(A)=0$ or $\mu(A^c)=0$. If
$\kk$ is irreducible, then either $R_0[\kk]>0$  or $k\equiv 0$ and $\Omega$ is an  atom of
$\mu$  in  $\cf$  (degenerate  case). A  simple  sufficient condition for irreducibility
is for the kernel to be positive a.e.

Let $\ca$ be  the set of $\kk$-invariant sets, and  notice that $\ca$ is
stable   by  countable   unions   and   countable  intersections.    Let
$\cfi=\sigma(\ca)$ be the $\sigma$-field  generated by $\ca$.  Then, the
operator $\kk$ restricted to an atom  of $\mu$ in $\cfi$ is irreducible.
We shall only consider non degenerate  atoms, and say the atom (of $\mu$
in $\cfi$)  is non-zero if the  restriction of the kernel  $\kk$ to this
atom  is non-zero  (and thus  the spectral  radius of  the corresponding
integral operator is  positive).  We denote by $(\Omega_i,  i\in I)$ the
at most countable  (but possibly empty) collection of  non-zero atoms of
$\mu$  in $\cfi$.   Notice that  the atoms  are defined  up to  an a.e.\
equivalence  and can  be  chosen to  be  pair-wise disjoint.   According
to~\cite[Lemma~5.3]{ddz-Re}, we have the decomposition:
\begin{equation}
   \label{eq:decomp}
  R_e[\kk]= \max_{i\in I}   R_e[\kk_i]
\quad\text{where}\quad
  \kk_i= \ind{\Omega_i} \kk \ind{\Omega_i}.
\end{equation}
We represent in Figure~\ref{fig:atomic1} an example of a kernel $\kk$
with its atomic decomposition using a ``nice'' order on $\Omega$ (so
the kernel is upper block triangular: the population on the left of an atom does not infect the
population on the right of an atom)
in Figure~\ref{fig:atomic2} the corresponding kernel
$ \kk'=\sum_{i\in I} \kk_i$; thanks to~\eqref{eq:decomp}, the
kernels $\kk$ and $ \kk'$ have the same effective reproduction
function: $R_e[\kk]=R_e[\kk']= \max_{i\in I} R_e[\kk_i]$.

\begin{figure}
  \begin{subfigure}[T]{.5\textwidth}
    \centering
    \includegraphics[page=1]{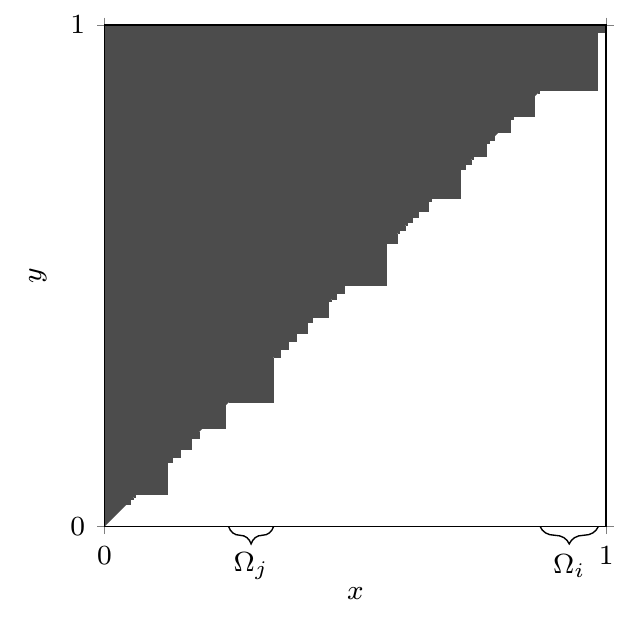}
    \caption{A representation of the kernel $\kk$ with the white zone included in
    $\{\kk=0\}$.}
    \label{fig:atomic1}
  \end{subfigure}%
  \begin{subfigure}[T]{.5\textwidth}
    \centering
    \includegraphics[page=2]{reducible}
    \caption{A representation of the kernel $\kk'=\sum_{i\in I}
      \kk_i$ with the white zone included in
    $\{\kk'=0\}$.}
    \label{fig:atomic2}
  \end{subfigure}
  \caption{Example of a  kernel $\kk$ on $\Omega=[0, 1]$  and the kernel
    $\kk'          =\sum_{i\in          I}         \kk_i$,          with
    $\kk_i(x,y)=\ind{\Omega_i}(x)\,  \kk(x,y)\,  \ind{\Omega_i}(y)$ and
    $(\Omega_i, i\in I)$ the non-zero atoms.  We
    have         $\spec(T_\kk)=\spec(T_{\kk'})$   as well as
    $R_e[T_\kk]=R_e[ T_{\kk'}]$.}
  \label{fig:atomic}
\end{figure}

We say the kernel $\kk$ is \emph{monatomic} if  there  exists a unique non-zero atom
($\sharp I=1$), and the kernel is \emph{quasi-irreducible} if it is monatomic, with
non-zero atom say $\oa$,  and $\kk\equiv 0$ outside $\oa\times \oa$. The quasi-irreducible
property is the  usual extension of the irreducible property in  the setting  of symmetric
kernels; and the monatomic property is the  natural generalization  to non-symmetric
kernels. We represented in Figure~\ref{fig:monatomic} a monatomic kernel $\kk$ with
non-zero atom say $\oa$ and in Figure~\ref{fig:quasi} the quasi-irreducible kernel
$\ka=\ind{\oa}\kk\ind{\oa}$ with the same atom; the set $\Omega$ being ``nicely ordered''
so that the representation of the kernels are upper triangular and the set $\oi$ in
Figure~\ref{fig:monatomic} corresponds to the sub-population infected by the atom $\oa$.

\begin{figure}
  \begin{subfigure}[T]{.5\textwidth}
    \centering
    \includegraphics[page=1]{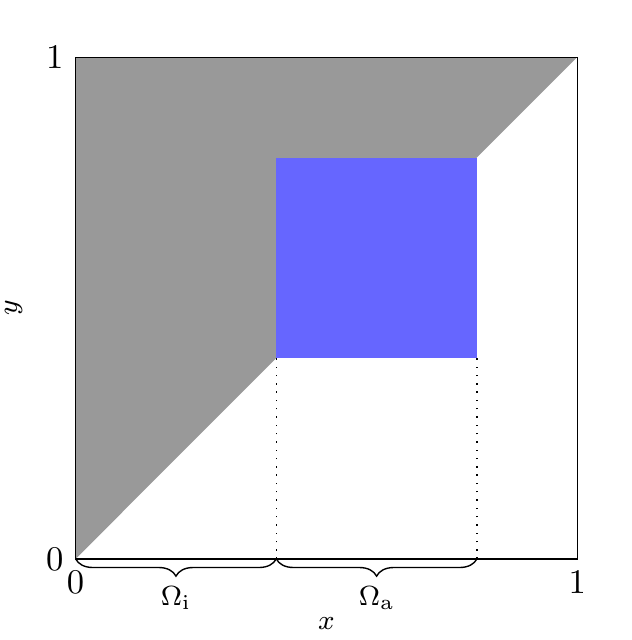}
    \caption{A representation of a monatomic kernel.}
    \label{fig:monatomic}
  \end{subfigure}%
  \begin{subfigure}[T]{.5\textwidth}
    \centering
    \includegraphics[page=2]{monatomic}
    \caption{A representation of a quasi-irreducible kernel.}
    \label{fig:quasi}
  \end{subfigure}
  \caption{Example of kernels $\kk$ and $\ka$ of a monatomic integral
    operator $T_\kk$ and the  quasi-irreducible
    integral operator $\Ta=T_{\ka}$ on $\Omega=[0, 1]$, with non-zero atom
    $\oa$. The kernels are zero on
    the white zone and  are irreducible when restricted to the blue zone.}
  \label{fig:monatomic-quasi}
\end{figure}

\subsection{The anti-Pareto frontier for irreducible and monatomic
  kernels}
\label{sec:q-irr}

We prove in the next result that for positive and/or irreducible
kernels, the gaps in~Table~\ref{tab:not-PAP} may essentially be
filled.  We illustrate these properties in Figure~\ref{fig:generic_frontiers}
by plotting the typical
Pareto and anti-Pareto frontiers for  irreducible kernels and positive
kernels. In order to avoid the degenerate  irreducible kernel,
we shall consider a  non-zero  kernel $\kk$, that is  a kernel such
that $\kk(\Omega,
\Omega)$ is positive.

\begin{proposition}[Consequences of irreducibility]
  \label{prop:k>0-c}
  %%% Prop
  Suppose that the cost function $C$ is continuous decreasing with $C(\un)=0$ and
  consider the loss function $R_e=R_e[\kk]$, with $\kk$ a finite double
  norm irreducible non-zero kernel. Then, we have the following properties:
  \begin{enumerate}[(i)]
  \item\label{item:k-mono}
    \begin{enumerate}[a)]
    \item\label{item:k-mono-mar}
      $R_0>0$.
      \item\label{item:k-mono-mar2} The function $\mar$ is continuous, decreasing
        on $[ \cmar, \cmax]$.
        \item The function $\Csup$ is continuous and decreasing
          on $[0, R_0]$.
        \item We have
$\Csup \circ \mar( c)=c$ for $c\in [ \cmar, \cmax]$.
\item\label{item:k-mono-AF}
  The set $\cpa$
is compact (for the weak-* topology), $\AF$
is connected and compact, and:
\[
  \AF = \{(c, \mar(c)) \, \colon \, c \in [\cmar, \cmax]\}.
\]

\item $\cmar=0$.
    \end{enumerate}
  \item\label{item:k>0} If furthermore $\kk>0$ a.e., then we also have:
    \begin{enumerate}[a)]
      \item  $\cmir=\cmax$.
    \item The
  strategy $\un$ (resp. $\zero$) is the only Pareto optimal as well as the only
  anti-Pareto optimal strategy with cost $c=0$ (resp. $c=1$).
 \end{enumerate}
\end{enumerate}
\end{proposition}

\begin{proof}
  According to \cite[Theorem~V.6.6]{schaefer_banach_1974}, if $\kk$ is
  an irreducible kernel with finite double norm, then, as $\kk$ is
  non-zero,  we have
  $R_0=R_0[\kk]>0$. This gives~\ref{item:k-mono}~\ref{item:k-mono-mar}.

  The other items follow from various results from~\cite{ddz-theo}: Assumptions~3 and 6
  from that paper hold, as well as Assumption~7, thanks to~\cite[Lemma~5.14]{ddz-theo}. In
  the notation of~\cite{ddz-theo}, as $\oa=\Omega$, we get $\cmar=C(\un)=0$. We conclude
  using \cite[Proposition~5.9]{ddz-theo} that
  items~\ref{item:k-mono}~\ref{item:k-mono-mar2}-~\ref{item:k-mono-AF} hold.

We now assume that $\kk>0$ a.e.  As  $\cmar=0$, we deduce that the strategy $\un$ is
anti-Pareto optimal. As $C$ is decreasing, we also get that the strategy $\un$ is Pareto
optimal.

Let $\eta\in \Delta$ be different from $\zero$. The kernel $\kk \eta$ restricted to the
set of positive $\mu$-measure $\{\eta>0\}$ is positive, thus the kernel $\kk \eta$
restricted to $\{\eta>0\}$ is positive. It is therefore irreducible and its spectral
radius is positive, so $R_e(\eta)>0$. This also readily implies that $\cmir=\cmax$ and
that the strategy $\zero$ is Pareto optimal. As $C$ is decreasing, we also get that the
strategy $\zero $ is anti-Pareto optimal.
\end{proof}

%  If furthermore $\kk$ is monatomic with atom $\oa$, then thanks to
% \cite[Lemma~5.13]{ddz-theo}, we have $ \cmar=C(\ind{\oa} )$ (which is 0
% if $\kk$ is irreducible); the function $\mar$ is continuous, decreasing
% on $[ \cmar, \cmax]$; the function $\Csup$ is continuous and decreasing
% on $[0, R_0]$; the functions $\mar$ and $\Csup$ are the inverse of each
% other, that is, $\mar \circ \Csup(\ell)=\ell$ for $\ell\in [0, R_0]$ and
% $\Csup \circ \mar( c)=c$ for $c\in [ \cmar, \cmax]$; and the set $\cpa$
% is compact and
% $ \AF = \{(c, \mar(c)) \, \colon \, c \in [\cmar, \cmax]\}$. \medskip

\begin{figure}[t]
  \begin{subfigure}[T]{.5\textwidth}
    \centering
    \includegraphics[page=1]{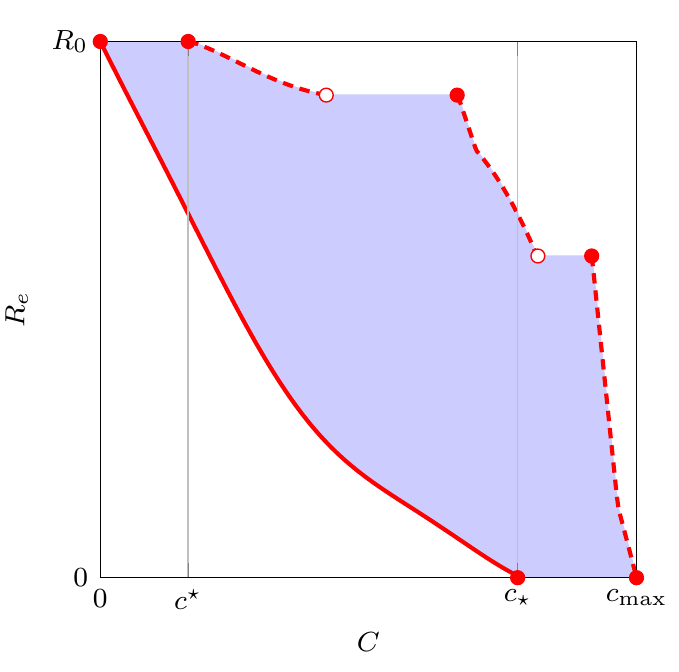}
    \caption{General kernel.}
    \label{fig:generic_frontiers-gen}
  \end{subfigure}%
  \begin{subfigure}[T]{.5\textwidth}
    \centering
    \includegraphics[page=2]{frontier}
    \caption{Monatomic kernel.}
  \end{subfigure}

  \begin{subfigure}[T]{.5\textwidth}
    \centering
    \includegraphics[page=3]{frontier}
    \caption{Irreducible kernel.}
    \label{fig:generic_frontiers-irr}
  \end{subfigure}%
  \begin{subfigure}[T]{.5\textwidth}
    \centering
    \includegraphics[page=4]{frontier}
    \caption{Kernel strictly positive almost surely.}
  \end{subfigure}
  \caption{Generic aspect of the feasible region (light blue), the
    Pareto frontier (thick red line) and the anti Pareto frontier
    (dashed red line) for the loss  function~$R_e[\kk]$, with kernel
    $\kk$, and a continuous decreasing cost function $C$.
    }
    \label{fig:generic_frontiers}
\end{figure}

\medskip

We now state the
properties of  the anti-Pareto frontiers for  monatomic kernel.

\begin{corollary}[Consequences of monatomicity]
  \label{cor:monat}
  Suppose that the cost function $C$ is continuous decreasing with $C(\un)=0$ and
  consider the loss function $R_e=R_e[\kk]$, with $\kk$ a finite double
  norm monatomic  kernel with non-zero atom $\oa$. Then,
  Properties~\ref{item:k-mono}~\ref{item:k-mono-mar}-\ref{item:k-mono-AF} of
  Proposition~\ref{prop:k>0-c} hold. The strategy
 $\ind{\oa}$ is anti-Pareto optimal with cost $\cmar=C(\ind{\oa})$.
\end{corollary}

\begin{proof}
   According to~\cite[Lemma~5.14]{ddz-theo}, we get that $R_0$ is
   positive and  $\cmar=C(\ind{\oa})$. The other results are proved as
   in Proposition~\ref{prop:k>0-c}.
\end{proof}
Using the
properties of  the anti-Pareto frontiers stated  in Proposition~\ref{prop:k>0-c}
for positive  kernels and  in Corollary~\ref{cor:monat}  for  monatomic
kernel, we plotted in Figure~\ref{fig:generic_frontiers}  the typical Pareto and
anti-Pareto  frontiers  for a  general  kernel  (notice the  anti-Pareto
frontier is not connected \textit{a priori}), a monatomic kernel (notice
the  anti-Pareto frontier  is  connected), and  a  positive kernel.

\subsection{Creating a \emph{cordon sanitaire} is not the worst idea}\label{sec:cordon}

We say a strategy $\eta\in \Delta$ is a \emph{cordon sanitaire} or \emph{disconnecting}
(for the kernel $\kk$) if $\eta\neq \zero$ and the kernel $\kk$ restricted to the set
$\{\eta>0\}$ is not connected (that is, not irreducible).
Let us first give a few elementary
comments on disconnecting strategies.
  \begin{itemize}
  \item The strategy $\eta=\un$ is disconnecting if and only if $\kk$
    is not connected.
  \item Disconnection  only depends  on fully vaccinated  individuals: A
    strategy  $\eta$  is  disconnecting  if and  only  if  the  strategy
    $\ind{\{\eta>0\}}$ is disconnecting.
  \item If $\kk>0$, then there is no disconnecting strategy.
    \item If $\eta\neq \zero$ is a strategy such that $\kk=0$ a.e.\ on
        $\{\eta>0\}^2$, then $\eta$ is disconnecting.
  \end{itemize}

The next proposition states that if the strategy $\eta$ is anti-Pareto optimal for a
kernel $\kk$ and non zero, then the kernel $\kk$ restricted to $\{\eta>0\}$ is irreducible. Let us
remark that in general this implication is not an equivalence.

\begin{proposition}[A \emph{cordon sanitaire} is never the worst idea]\label{prop:cut}
  %%% PROPOSITION
  Suppose that the cost function~$C$ is continuous decreasing and consider the loss
  function $R_e[\kk]$, with $\kk$ a finite double norm kernel on $\Omega$ such that
  $R_0[\kk]>0$. Then, a disconnecting strategy is not anti-Pareto optimal.
\end{proposition}

In the non-oriented cycle graph from Example~\ref{ex:cycle-graph}, this property is
illustrated in Figure~\ref{fig:perf} as the disconnecting strategy ``one in $4$'' is not
anti-Pareto optimal; see Figure~\ref{fig:cycle-kern-graph-disc}

\begin{proof}

  Let $\eta$ be a disconnecting strategy, and thus $\eta\neq \zero$. Since $\eta$ is
  disconnecting, that is, $\kk$ restricted to $\{\eta>0\}$ is not irreducible, we deduce
  there exists $A, B\in \cf$ such that $\mu(A)>0$, $\mu(B)> 0$, $(\kk \eta) (B, A)=0$ and
  a.e.\ $A\cup B=\{\eta>0\}$ and $A\cap B=\emptyset$. We deduce
  from~\cite[Equation~(29)]{ddz-Re} where we can replace  $\kk$  by $\kk \eta$ that:
  \begin{equation}
    \label{eq:R0=maxR0}
    R_e[\kk\eta](\ind{A}+ \ind{B})=\max \left( R_e[\kk\eta](\ind{A}),R_e[\kk\eta](\ind{B}) \right).
  \end{equation}
  First assume that $R_e[\kk\eta](\ind{A})\geq R_e[\kk\eta](\ind{B})$, so that:
  \[
    R_e[\kk](\eta)=   R_e[\kk\eta](\ind{\{\eta>0\}})=
    R_e[\kk\eta](\ind{A}+\ind{B})= R_e[\kk\eta](\ind{A}).
  \]
  For $\theta\in [0, 1]$, define the strategy $\eta_\theta=\eta \ind{A}+ \theta \eta
  \ind{B}$. We deduce that:
  \begin{align*}
    R_e[\kk](\eta_\theta) = R_e[\kk\eta_\theta](\ind{A} + \ind{B}) &= \max
    (R_e[\kk\eta_\theta](\ind{A}), R_e[\kk\eta_\theta](\ind{B})) \\
								   &=\max (R_e[\kk\eta](\ind{A}), \theta R_e[\kk\eta](\ind{B}))\\
								   & = R_e[\kk\eta](\ind{A})\\
								   &= R_e[\kk](\eta),
  \end{align*}
  where we used \eqref{eq:R0=maxR0} with $\eta$ replaced by $\eta_\theta$ for the second
  equality as $(\kk \eta_\theta) (B, A)=0$, and the homogeneity of the spectral radius in
  the third. Thus, the map $\theta\mapsto R_e[\kk](\eta_\theta)$ is constant on $[0, 1]$.
  Since $\mu(B)>0$ and $C$ is decreasing, we get that $\theta \mapsto C(\eta_\theta)$ is
  decreasing. This implies that $\eta_\theta$ is worse than $\eta$ for any $\theta\in [0,
  1)$, and thus $\eta$ is not anti-Pareto optimal.

  The case $R_e[\kk\eta](\ind{B})\geq R_e[\kk\eta](\ind{A})$ is handled similarly.
\end{proof}

\begin{remark}\label{rem:plateau}
  %%% REMARK
  If the kernel $\kk$ is irreducible and non-zero, then the upper
  boundary of the set of outcomes $\FF$
  is the anti-Pareto frontier; see Figure~\ref{fig:generic_frontiers-irr} for instance.
  We deduce from Proposition~\ref{prop:cut} that if $\eta_0$ is a disconnecting strategy,
  then  we have that $R_e[\kk](\eta_0)$ is strictly less that $
  \sup\{R_e[\kk](\eta)\, \colon\, C(\eta)=C(\eta_0)\}$.

  However, if the kernel $\kk$ is not irreducible, then the trivial strategy $\un $ is
  disconnecting. Furthermore, the upper boundary of the set of outcomes $\FF$ is not
  reduced to the anti-Pareto frontier; see Figure~\ref{fig:generic_frontiers-gen} for
  instance. In fact, there exists disconnecting strategies that are not anti-Pareto
  optimal, but whose outcomes lie on the flat parts of the upper boundary of $\FF$. In
  particular, such strategies have the worst loss given their cost. However, it is not
  difficult to check that they do not disconnect further than the trivial strategy $\un$.
\end{remark}

\subsection{Pareto and anti-Pareto frontiers for reducible kernels}\label{sec:reducible}

Let us now  assume that the kernel $\kk$ is  ``truly reducible'', in the
sense that  it has  at least two  non-zero atoms, and thus $R_0=R_0[\kk]>0$.  We  will see  in this
section how to  effectively reduce the study of  the global optimization
problem to  a study of the  optimization problem on each  non-zero
atom. Recall the collection of non-zero atoms $(\Omega_i, i\in I)$
defined in Section~\ref{sec:atom} and the corresponding
quasi-irreducible kernels $(\kk_i, i\in I)$ in~\eqref{eq:decomp}.
By construction, the kernel  $\kk_i$ has a finite double norm and
$R_0[\kk_i]>0$.

We now describe two ways of restricting the problem to an atom.
For the kernel  $\kk_i$ and the loss function $R_e[\kk_i]$,  the atom is
still viewed as  a part of the larger population~$\Omega$. As such, the
vaccination  strategies   that  agree   on  $\Omega_i$  but   differ  on
$\Omega_i^c$ will have  the same loss, but  their costs may
differ.
For $i\in I$ and  $\eta\in \Delta$, we
set similarly:
\[
  \eta_i=\eta \ind{\Omega_i}.
\]
% We now describe two ways of restricting the problem to an atom.
% For $i\in I$, we set:
% \begin{equation}\label{eq:ki}
%   \kk_i=\ind{\Omega_i} \kk \ind{\Omega_i}.
% \end{equation}
% This is a quasi-irreducible kernel on $\Omega$ with finite double
% norm and $R_0[\kk_i]>0$. The first equality in
% Equation~\eqref{eq:R=maxRi2} below gives  the decomposition of $R_e[\kk]$
% using $R_e[\kk_i]$.
% We set $R_0=R_0[\kk]$ and assume it is
% positive.
We consider the loss $R_e[\kk_i]$ and the
corresponding optimal loss function $R^\star_i$ defined on
$[0, \cmax]$ and optimal cost function $C^\star_i$ and $C_{i, \star}$. For convenience
the functions $C^\star_i$ and $C_{i, \star} $ which are defined on $[0, R_0[\kk_i]]$ are
extended to $[0, R_0]$ by letting them be equal to 0 on
$(R_0[\kk_i], R_0]$.  %Notice also that $\{\kk_i\equiv 0\}=\Omega_i^c$.

Another  point  of  view  is  to restrict  the  kernel  and  vaccination strategies   to
the atom, and study it intrinsically, in isolation. Quantities  and functions
defined by  this intrinsic approach will     be    denoted     by    bold     letters.
In particular $\bm{\kk}_i:\Omega_i^2 \to \mathbb{R}$ is the kernel $\kk$ (and $\kk_i$)
restricted to $\Omega_i$; it is irreducible and non-zero by construction and
$R_0[\bm{\kk}_i]$ is a simple positive eigenvalue of the corresponding  integral operator.
If $\eta$ is a  vaccination  strategy,  then  $\bm{\eta}_i$  is  its  restriction  to
$\Omega_i$.  By construction, we have for all $\eta\in \Delta$:
\[
  R_e[\kk_i](\eta)= R_e[\kk_i](\eta_i)=R_e[ \bm{\kk}_i](\bm{\eta}_i).
\]

If $\bm{\eta}$ is a $[0,1]$-valued measurable function defined on $\Omega_i$, we define
its extension $\eta$ on $\Omega$ (corresponding to no  vaccinations outside $\Omega_i$)
and its cost by:
\[
  \eta=\begin{cases}
    \bm{\eta} & \text{on $\Omega_i$}\\
    \un & \text{on $\Omega_i^c$}\\
  \end{cases}
  % \eta(x) = \un_{\Omega_i^c}(x) + \bm{\eta}(x)\un_{\Omega_i}(x)
  \quad  \quad\text{and}\quad
  \bm{C}_i (\bm{\eta}) = C(\eta).
\]
The optimization  problems \eqref{eq:bi-min}  and~\eqref{eq:bi-max} may now be stated on
each $\Omega_i$  for the kernel $\bm{\kk}_i$, the loss $R_e[\bm{\kk}_i]$ and the cost
$\bm{C}_i$: denote by $\bm{C}_{i,\star}$ and  $\bm{C}_i^\star$ the  corresponding  optimal
cost functions,  and extend  them  to  $[0,  R_0]$  by   letting  them  be  equal  to  0
on $(R_0[\mathbf{k}_i], R_0]$. In particular, by construction, $\bm{C}_{i,\star}$ is equal
to $C_{i,\star}$.  However, there  is   no  relation  in  general between $\bm{C}_i^\star$
and $C_i^\star$.  Nevertheless, it is possible to establish  such a relation when  the
cost is extensive.  Recall once more  that  for  a  vaccination  strategy  $\eta$,  the
proportion  of vaccinated individuals of trait $x$  is given by $1-\eta(x)$. Thus, two
vaccination strategies  $\eta$ and  $\eta'$ target disjoint  subsets of the    population
if $\eta\vee\eta'=\un$.

\begin{definition}[Extensivity] Let $C$ be  a continuous decreasing cost
  function with $C(\un)=0$.  The cost $C$ is  called \emph{extensive} if
  vaccinating   disjoint  subsets   of  the   population  is   additive:
  \[  C(\eta\wedge\eta')=C(\eta)+C(\eta') \quad \text{for all $\eta,  \eta'\in
      \Delta$ such that $ \eta \vee \eta'= \un$}.
  \]
\end{definition}

If the continuous decreasing cost function $C$  is extensive, then  we get for all
$\eta\in \Delta$ that:
\begin{equation}
  \label{eq:ext-cost}
  C(\eta)=\sum_{i\in I} C(\eta_i+ \ind{\Omega_i^c})=\sum_{i\in I} \bm{C}_i(\bm{\eta}_i),
\end{equation}
since all the vaccinations  $\eta_i+ \ind{\Omega_i^c}$ target pairwise disjoint subsets of
the population.

\begin{remark}[Affine costs are extensive]
  %%% REMARK
  If       the       cost       function      takes       the       form
  \[
    C(\eta)  =  \cmax  -  \int_\Omega \phi(\eta(x),x) \,\mu(dx)
  \]
  where $\phi \, \colon \, [0,1] \times \Omega \to \R_+$ is measurable and non-decreasing
  in its first variable, then $C$  is extensive. In particuar, the  affine cost  functions
  considered  in Proposition~\ref{prop:critical-ray} are extensive.
\end{remark}

We are  now ready  to state  the reduction  result, which  in particular
implies that if  the cost function is extensive,  then the (anti-)Pareto
frontier  of the  full  model  may be  constructed  from  the family  of
(anti-)Pareto frontiers of each atom.

 \begin{proposition}[Reduction to atoms]\label{prop:decoupling}
   Let $\kk$ be a kernel with finite double norm on $\Omega$, such that
   $R_0 = R_0[\kk]>0$.
   Suppose that the cost function $C$ is continuous decreasing with $C(\un)=0$.
   \begin{enumerate}[(i)]
     \item \textbf{Decomposition of the loss.} \label{it:loss}
   For any $\eta\in\Delta$, we have:
\begin{equation}\label{eq:R=maxRi2}
  R_e[\kk](\eta)
  =\max_{i\in I} R_e[\kk](\eta \ind{\Omega_i})
  =\max_{i\in I} R_e[\kk_i](\eta_i)
  =\max_{i\in I} R_e[\bm{\kk}_i](\bm{\eta}_i).
\end{equation}

 \item \textbf{Anti-Pareto optimal strategies.}
   \label{it:anti-p}
  For all $\ell\in [0,
   R_0]$ and $\eta \in \Delta$, the  following two properties are equivalent:
  \begin{enumerate}[a)]
  \item\label{it:AP1}
    The strategy $ \eta$ is anti-Pareto optimal with $
        R_e[\kk](\eta) = \ell$.
      \item\label{it:AP2}  There exists $j \in
        \mathrm{argmax}_{i\in I} \, C_i^\star(\ell)$ such that
$\eta=0$ on $\Omega_j^c$ and $\eta=\bm{\eta}_j$ on
        $\Omega_j$, where   $\bm{\eta}_{j}$ is
        anti-Pareto optimal for $\bm{\kk}_j$ on $\Omega_{j}$ and cost
        function $\bm{C}_i$ with
        $R_e[\bm{\kk}_j](\bm{\eta}_j) =\ell$.
   \end{enumerate}
     Besides, we have:
 \begin{equation}
   \label{eq:R*C*}
    R_e^\star= \max_{i\in I}\, R^\star_i \quad\text{on $[0, \cmax]$}\quad\text{and}\quad
    \Csup= \max_{i\in I}\, C^\star_i \quad\text{on $[0, R_0]$}.
      \end{equation}
Furthermore, if the cost function $C$ is extensive, then  for all $i\in I$, we
  have:
  \[
    C_i^\star=
    \bm{C}_i^\star + C(\un_{\Omega_i}).
  \]

\item  \textbf{Pareto  optimal  strategies  when the  cost  function  is
    extensive.}\label{it:pareto} Suppose  that the cost function  $C$ is
  extensive.  For all  $\ell\in [0,  R_0]$  and $\eta  \in \Delta$,  the
  following two properties are equivalent:
  \begin{enumerate}[a)]
  \item\label{it:P1} The strategy $\eta$ is Pareto optimal with $ R_e[\kk](\eta) = \ell$.
  \item\label{it:P2} On $(\bigcup_{i\in
             I} \Omega_i)^c$, $\eta=1$  and, for all $i\in I$, $\eta$
           restricted to 
           $\Omega_i$, say  $\bm{\eta}_i$, is   Pareto optimal
           for $\bm{\kk}_i$ on $\Omega_i$  and cost
        function $\bm{C}_i$ with  $R_e[\bm{\kk}_i](\bm{\eta}_i) =
           \min(\ell, R_0[\bm{\kk}_i])$ (and thus $\bm{\eta}_i=\un$ if
           $R_0[\bm{\kk}_i]\leq  \ell$).
   \end{enumerate}
   Besides, we have:
   \[
     C_\star = \sum_{i\in I} \bm{C}_{i,\star}.
   \]
    \end{enumerate}
\end{proposition}

\begin{remark}[Additional consequences]
  %%% REMARK
  From~\eqref{eq:FL=L*} and the second part of~\eqref{eq:R*C*}, we get that the
  anti-Pareto frontier is given by: \[ \AF=\left\{\left(\max_{i\in I} C_i^\star(\ell),
  \ell)\right)\, \colon\, \ell \in [0, R_0]\right\}. \] We   deduce from
  Point~\ref{it:anti-p} that the maximal cost of totally inefficient strategies is given
  by:
  \[
    \cmar:=\Csup(R_0)=\max_{ i\in I} \{C(\ind{\Omega_i})\, \colon\,
    R_0[\kk_i]=R_0[\kk]\}.
  \]
  According to \cite[Remark~5.1(v)]{ddz-Re}  the number of  atoms $\Omega_i$ such that
  $R_0[\kk_i]=R_0[\kk]$  is   equal  to  the  algebraic  multiplicity of  $R_0$  for
  $T_\kk$.

  As any Pareto optimal strategy is larger than $\ind{(\bigcup_{i \in I} \Omega_i)^c}$
  according to Point~\ref{it:pareto},  we get  an upper  bound for  the minimal  cost
  which ensures that no infection occurs at all:
  \[
    \cmir=\Cinf(0)\leq C(\eta) \quad\text{with}\quad
    \eta=\un - \sum_{i\in I} \ind{\Omega_i}.
  \]
\end{remark}

\begin{remark}\label{rem:L-discont}
  %%% REMARK
  If  $R_0[\kk]>0$ and  $\kk$ is  not  monatomic, then  Assumption 7  in
  \cite{ddz-theo} (that  is any  local maximum of  the loss  function is
  also  a global  maximum) may  or  may not  be satisfied  for the  loss
  function  $\loss=R_e[\kk]$;  this   can  happen even   in  a  two
  homogeneous populations  model.  In  the former  case the  function $C  ^\star$ is
  continuous and the  anti-Pareto frontier is connected,  whereas in the
  latter  case the  function  $C ^\star$  may have  jumps  and then  the
  anti-Pareto frontier has more than one connected component.
\end{remark}

\begin{proof}[Proof of Proposition~\ref{prop:decoupling}]
Let   $\kk$ be  a finite double norm kernel
on $\Omega$ such that $R_0=R_0[\kk]>0$.  Set
$\Omega_0=\Omega   \setminus  \cup_{i\in   I}   \Omega_i$.
For $i\in I$ and $\eta\in \Delta$, we set $\eta_i=\eta
\ind{\Omega_i}$.

According to~\eqref{eq:decomp} and since $R_e[\kk_i](\eta)=
R_e[\kk_i](\eta_i)=R_e[\kk](\eta_i)$,  we can
decompose $R_e[\kk]$ according to the quasi-irreducible components $(\kk_i, i\in I)$ of
$\kk$ to get that   for $\eta\in \Delta$:
\begin{equation}\label{eq:R=maxRi}
  R_e[\kk](\eta)
   =\max_{i\in I} R_e[\kk_i](\eta)
 =\max_{i\in I} R_e[\kk_i](\eta_i)
 =\max_{i\in I} R_e[\kk](\eta_i).
\end{equation}
Then use that $\bm{\kk}_i$ is the restriction of $\kk_i$ to $\Omega_i$ to
get  Point~\ref{it:loss}.

We  now prove  Point~\ref{it:anti-p}.  Equation~\eqref{eq:R=maxRi}  and the     definition
of $R_e^\star$    readily     implies    that $ R_e^\star=  \max_{i\in I}\,
R^\star_i$, which  gives the  first part of~\eqref{eq:R*C*}.

We prove  that properties~\ref{it:AP1} and~\ref{it:AP2}  are equivalent.
The case  $\ell=0$ being trivial,  we only consider $\ell\in  (0, R_0]$.
Let $\eta$ be a strategy  such that $R_e[\kk](\eta)=\ell$.  According to
\ref{it:loss},       there       exists        $j$       such       that
$R_e[\kk](\eta)=R_e[\kk](\eta_j)$. Since $\ell>0$,  we get that $\eta_j$
is not equal to $\zero$. Hence, we get:
 \[
   C(\eta) \leq \inf_{i \in I} C(\eta_i) \le C(\eta_j) \leq C^\star_j(\ell) \leq \sup_{i \in
   I} C^\star_i(\ell) \leq C^\star(\ell),
 \]
 where:
 \begin{enumerate}
   \item the first and second inequalities become equalities if and only if $\eta_i =
     \zero$ for all $i \neq j$ because $C$ is decreasing;
   \item the third inequality is an equality if and only if
     $\bm{\eta}_j$ is anti-Pareto 
     optimal (see Table~\ref{tab:not-PAP});
   \item the last inequality follows from the fact that $R_e[\kk_i](\eta_i) =
     R_e[\kk](\eta_i)$ for all $i \in I$.
 \end{enumerate}
 Hence, Property~\ref{it:AP1} is equivalent to the following equalities:
 \begin{equation}
   \label{eq:i=star} C^\star_j(\ell)=C(\eta_j)=C(\eta)=C^\star(\ell).
 \end{equation}
 which is equivalent to Property~\ref{it:AP2}. In particular, it follows from the
 existence of the anti-Pareto optimal strategy that $\sup_{i\in I} C^\star_i$ is in fact a
 $\max$.

We now prove  that $ C_i^\star= \bm{C}_i^\star  + C(\un_{\Omega_i})$ for
all $i\in I$ in case $C$ is extensive.   Note that the optimal cost  $\bm{C}_i^\star$,
defined in terms of the restricted kernel $\bm{\kk}_i$,  and which may be viewed as
intrinsic on $\Omega_i$,  differs from the cost  $C_i^\star$, defined on the  
``extrinsic''  kernel   $\kk_i$   defined  on   the  whole   space $\Omega$. Let $\ell\in 
[0, R_0]$.  The worst vaccinations  on the whole space   clearly consist in  vaccinating
everyone outside $\Omega_i$ and vaccinating in the  worst possible way inside $\Omega_i$,
that is,  if $\eta$ is anti-Pareto  optimal for the kernel  $\kk_i$ with loss    $\ell\in 
[0,    R_0]$     and    cost     $C(\eta)$,    then $\eta=\bm{\eta}_i  \ind{\Omega_i}$,
where  $\bm{\eta}_i$ is  anti-Pareto optimal  for  the   kernel  $\bm{\kk_i}$  with  loss 
$\ell$  and  cost $\bm{C}_i^\star(\ell)=\bm{C}_i(\bm{\eta}_i)$.                        Set
$\eta'=\eta+\ind{\Omega_i^c}$   and   $\eta''=\ind{\Omega_i}$  so   that $\eta=\eta'
\vee \eta''=\un$. By definition of $\bm{C}_i$, we have
$C(\eta')=\bm{C}_i(\bm{\eta}_i) $.  Since $C$ is extensive, we get:
\[
  C^\star_i(\ell)=C(\eta)=C(\eta')+ C(\eta'')=\bm{C}_i(\bm{\eta}_i)  +
  C(\ind{\Omega_i})= \bm{C}_i^\star(\ell) +
  C(\ind{\Omega_i}).
\]

 Point~\ref{it:pareto} follows directly from Point~\ref{it:loss} and the the following
 decomposition of $C$ as an extensive function:
 \begin{equation}
   C(\eta) = \sum_{i \in I} \bm{C}_i(\bm{\eta}_i). \qedhere
 \end{equation}

\end{proof}

\printbibliography
\end{document}